\documentclass[11pt]{amsart}
\usepackage{a4}
\usepackage{amssymb}
\usepackage{stackengine}

\stackMath

\usepackage{color}
\makeatletter
\newcommand{\leqnomode}{\tagsleft@true}
\newcommand{\reqnomode}{\tagsleft@false}
\makeatother

\newtheorem{theorem}{Theorem}[section]
\theoremstyle{plain}
\newtheorem{lemma}[theorem]{Lemma}
\newtheorem{proposition}[theorem]{Proposition}

\newtheorem{corollary}[theorem]{Corollary}

\newtheorem{remark}[theorem]{Remark}

\newtheorem{definition}[theorem]{Definition}

\numberwithin{equation}{section}

\DeclareMathOperator*{\esssup}{ess\,sup}

\DeclareMathOperator*{\essinf}{ess\,inf}

\newcommand\trnorm[1]{\left\|\kern-1.2pt\left|#1\right\|\kern-1.2pt\right|}

\newcommand{\N}{\mathbb{N}}
\newcommand{\C}{\mathbb{C}}
\newcommand{\R}{\mathbb{R}}

\DeclareMathOperator*{\supp}{supp}
\DeclareMathOperator*{\esssupp}{ess\,supp}

\let\latexchi\chi
\makeatletter
\renewcommand\chi{\@ifnextchar_\sub@chi\latexchi}
\newcommand{\sub@chi}[2]{
  \@ifnextchar^{\subsup@chi{#2}}{\latexchi^{}_{#2}}%
}
\newcommand{\subsup@chi}[3]{
  \latexchi_{#1}^{#3}%
}
\makeatother
\begin{document}

\title[Geometric properties of Sobolev spaces $W^{1,\Phi}$]{Uniform convexity, reflexivity, supereflexivity and $B$ convexity  of generalized Sobolev spaces $W^{1,\Phi}$}



\author{Anna Kami\'{n}ska}
\address{Department of Mathematical Sciences,
The University of Memphis, TN 38152-3240}
\email{kaminska@memphis.edu}

\author{Mariusz \.Zyluk}
\email{mzyluk@gmail.com}

\date{\today}

\thanks{}

\begin{abstract} We investigate Sobolev spaces $W^{1,\Phi}$ associated to Musielak-Orlicz  spaces $L^\Phi$.  We first present conditions  for the boundedness of the Voltera operator in $L^\Phi$.   Employing this,   we provide necessary and sufficient conditions for   $W^{1,\Phi}$  to contain isomorphic subspaces to $\ell^\infty$ or  $\ell^1$. Further we give necessary and sufficient conditions in terms of the function $\Phi$ or its complementary function $\Phi^*$ for reflexivity,    uniform convexity, $B$-convexity and superreflexivity of  $W^{1,\Phi}$. As corollaries we obtain the corresponding results for Orlicz-Sobolev  spaces $W^{1,\varphi}$ where $\varphi$ is an Orlicz function, the variable exponent Sobolev spaces $W^{1,p(\cdot)}$  and the Sobolev spaces associated to double phase functionals.
\end{abstract}

\maketitle

\section{Introduction}

  The main goal of this paper is to study  geometric properties of Sobolev spaces $W^{1,\Phi}$ induced by  Musielak-Orlicz spaces $L^\Phi$, where $\Phi$ is an Orlicz function with parameter, called also  Musielak-Orlicz function.

Musielak-Orlicz spaces appeared first time in the literature in 1951 in H. Nakano paper \cite{nakano}, and later  J. Musielak and W. Orlicz in 1959 in the paper  \cite{musielak orlicz} gave a more general definition often more suitable for applications.  Musielak-Orlicz spaces engendered some interest and were extensively studied during the seventies, eighties and nineties of the last century by various groups of mathematicians across the world. In particular the structural and geometrical properties of those spaces were well understood.

 On the other hand, the Musielak-Orlicz-Sobolev spaces ($MOS$ spaces) came to the light late in seventies. The first results about $MOS$ spaces were established by H. Hudzik in series of paper between 1976-1979 (see e.g. \cite{hudzik 1979}). We remark here that the author needed to assume some rather strong assumptions about the function $\Phi$ to establish his results, which could be  common in the field of $MOS$ spaces.

 Parallel to the research on $MO$ and $MOS$ spaces the variable exponent spaces, also called the Nakano spaces, were of significant interest. In terms of MO spaces a variable exponent space $L^{p(\cdot)}$ is a space generated by a function $\Phi$ of the form ${t^{p(x)}}/{p(x)}$, where $p(x)$ is a measurable function such that $p(x)\geq 1$. Surprisingly, a serious investigation into Sobolev spaces based on $L^{p(\cdot)}$ began only in the nineties of the last century with the paper of O. Kov\'acik and J. R\'akosník \cite{pepiczki}, where the authors proved some basic properties of variable exponent Sobolev spaces.

After the initial interest in the $MOS$ spaces, the research in the area  remained dormant for almost a decade. Surprisingly, the research was reinvigorated by the interest of those spaces for their application in physics. 
One of the  applications of $MOS$ came from modeling electrorheological fluids 
 - fluids whose viscosity changes in the presence of an electrical field. In 2000 M. Ruzicka \cite{ruzicka} provided a model for mechanics of those fluids that employs the variable exponent Sobolev spaces. In 2002 L. Diening in his dissertation \cite{diening dis} expanded  the theory of Ruzicka and provided, among other things, the sufficient condition on  the regularity of the exponent $p(x)$ to guarantee the boundedness of Hardy-Litllewood maximal operator.
Those results rekindled the interest in $MOS$ spaces and motivated other authors to study partial differential equations in the context of $MOS$ spaces \cite{CGSW}.

The paper consists of  six sections. In the introductory part we define several notions related to Banach function spaces, Musielak-Orlicz spaces $L^\Phi$, their norms and useful inequalities of Musielak-Orlicz functions ($MO$ functions) $\Phi$. We also consider variable exponent  and double phase $MO$ functions showing  necessary and sufficient conditions for those functions and their conjugates to satisfy the growth condition $\Delta_2$. At the end we define Musielak-Orlicz-Sobolew spaces ($MOS$ spaces) $W^{1,\Phi}$ on a finite interval $(\alpha, \beta)$.

In the second section we study integral operators between Musielak-Orlicz spaces. In particular we obtain a characterization of the bounded Voltera operators on  $L^\Phi$ under the assumption of so called (V) condition. It appears that (V) condition is always satisfied in Orlicz spaces, variable exponent Lebesgue spaces  as well as in the spaces induced by double phase functionals. 

The third section is devoted to characterization of the Sobolev spaces $W^{1,\Phi}$ containing an isomorphic subspace of $\ell^\infty$.  Before that we complete the analogous results in $MO$ spaces and in particular in variable exponent Lebesgue spaces.  It appears that the lack of the growth condition $\Delta_2$ of $\Phi$ is a sufficient condition for $W^{1,\Phi}$ to contain of $\ell^\infty$ and is also necessary whenever condition (V) is satisfied. 

In the fourth section we do analogous investigations concerning the existence of a subspace isomorphic to $\ell^1$ in $W^{1,\Phi}$. We obtain that if either $\Phi$ or $\Phi^*$ does not satisfy $\Delta_2$ then $W^{1,\Phi}$ contains such a subspace. The necessity of this occurs under condition (V). It follows  complete characterizations of the containment of $\ell^1$ in Orlicz-Sobolev spaces $W^{1,\varphi}$,  variable exponent Lebesgue spaces $L^{p(\cdot)}$ or variable exponent Sobolev spaces $W^{1,p(\cdot)}$. 

The short fifth section states necessary and sufficient conditions on reflexivity of $W^{1,\Phi}$ followed by corresponding corollaries in $W^{1,p(\cdot)}$ and $W^{1,\varphi}$.

Section sixth is the main part of this paper. Here we characterize uniform convexity of $W^{1,\Phi}$. It is expressed in terms of $\Phi$ and its conjugate $\Phi^*$. Since $W^{1,\Phi}$ is an isometric subspace of the  product $L^\Phi\times L^\Phi$, the conditions for  uniform convexity for $L^\Phi$ are sufficient for uniform convexity of $W^{1,\Phi}$. The key result here is to show that those conditions, $\Delta_2$ of $\Phi$ and uniform convexity of $\Phi$, are also necessary. In this part we also give a complete characterization of uniform convexity of $L^{p(\cdot)}$ as well as $W^{1,p(\cdot)}$ and $W^{1,\varphi}$.

In the last section seven, on the basis of  the previous results we  show that under condition (V), reflexivity, superreflexivity and $B$-convexity in $W^{1,\Phi}$ are equivalent, and they hold whenever $\Phi$ and $\Phi^*$ satisfy $\Delta_2$.   We obtain analogous results in Orlicz-Sobolev spaces $W^{1,\varphi}$ and variable exponent Sobolev spaces $W^{1,p(\cdot)}$. 

All results contained in this paper are proved for $MO$ space $L^\Phi$ or $MOS$ space $W^{1,\Phi}$,  where   $\Phi$ is a general $MO$ function occasionally with mild assumption.  In the particular case where $\Phi$ is a variable exponent function, most results  hold true without any additional assumptions, and have not been known before.

  Let further $\R$ be the set of real numbers, $\N$ the set of natural numbers and  $ \R_+=[0,\infty) $.    Let $(\Omega, \Sigma, \mu)$ be a measurable  space, where $\Sigma$ is a $\sigma$-algebra of subsets of $\Omega$ and $\mu$ is a $\sigma$-finite, complete measure on $\Sigma$.  By $L^0 = L^0(\Omega)$ denote the set of all $\mu$-measurable complex valued functions on $\Omega$. Recall that $(X, \|\cdot\|_X)$ is a {\it Banach function space} if $X\subset L^0$, and if $f\in L^0$, $g\in X$ and $|f| \le |g|$ a.e. then $f\in X$ and $\|f\|_X \le \|g\|_X$. We say that  a Banach function space  $(X, \|\cdot\|_X)$ has the  {\it Fatou property},  if for any $0\le f_n \uparrow f$ a.e., $f_n\in X$, $f\in L^0$ and $\sup_n\|f_n\|_X<\infty$ then $f\in X$ and $\|f_n\|_X \uparrow \|f\|_X$ as $n\to\infty$. An element $f\in X$ is called {\it order continuous} whenever for any $0\le f_n \le f$ with $f_n\downarrow 0$ a.e., we have $\|f_n\|_X \downarrow 0$. Letting $X_a$ be the set of all order continuous  elements from $X$, and $X_b$ be the closure of all simple functions from $X$, we have $X_a\subset X_b$. Let $X^*$ be the dual space to $X$. The {\it K\"othe dual space}  $X'$ of $X$  \cite{BS, Z} is defined as follows
\[
X' = \left\{ f\in L^0:  \|g\|_{X'} = \sup\left\{\int_\Omega fg\,d\mu\, : \ \|f\|_X \le 1\right\} < \infty \right\} .
\]
The space $X'$  equipped with the norm $\|\cdot\|_{X'}$ is  a Banach function space satisfying the Fatou property. If $X_a = X_b$ and $X$ has the Fatou property then $(X_a)^*$ is isometrically isomorphic to $X'$. In this case $X^* \simeq X' \oplus X_s^*$, where the symbol $\simeq$ denotes linear isometry, and $X_s^*= X_a^\perp$ is the set of all singular functionals that is the set of $S\in X^*$ such that $S(f) = 0$ for every $f\in X_a$. For references on function spaces see \cite{BS, Kant, LT2, Lux, Z}.

  A function $\varphi : [0,\infty) \to [0,\infty]$ is called an {\it Orlicz function with extended values}, if $\varphi$ is not identically 0, $\lim_{t\to 0+} \varphi (t) =
\varphi (0) = 0$, and $\varphi$ is left continuous and convex on $(0, b_\varphi]$, where $b_\varphi = \sup\{t > 0 : \varphi(t) < \infty\}$. If for every $t\ge 0$, $\varphi(t) < \infty$, then $\varphi$ is called an {\it Orlicz function} \cite{KR, Lux}. 

 A function $\Phi: \Omega\times  [0,\infty) \to [0,\infty]$ is called a {\it Musielak-Orlicz function with extended values} ($eMO$ function for short) if for a.a. $x\in \Omega$, $\Phi(x, \cdot)$ is an
Orlicz function with extended values and for all $t \ge 0$, $\Phi(\cdot, t)$ is measurable.   
  If in addition  $\Phi(x, t) <\infty$ for a.a. $x\in\Omega$, $t\ge 0$, then it is called a {\it Musielak-Orlicz} function ($MO$ function for short).
  Given a Musielak-Orlicz  function $\Phi$ with extended values, the {\it Musielak-Orlicz space} ({\it $MO$ space}), called also {\it generalized Orlicz space}  $L^\Phi = L^\Phi(\Omega)$, consists of all functions $f\in L^0$ such that 
\[
I_\Phi(\lambda f) = \int_\Omega\Phi(x,\lambda |f(x|)\,d\mu(x) < \infty 
\]
for some $\lambda>0$.
 The {\it Luxemburg norm}
\[
\|f\|_\Phi = \inf\{ \lambda > 0: I_\Phi(f/\lambda) \le 1\},
\]
and the {\it Orlicz norm}
\[
\|f\|_\Phi^0 = \sup_{I_{\Phi^*}(g) \le 1} \int_\Omega f(x) g(x) d\mu(x) =\sup_{I_{\Phi^*}(g) \le 1} \int_\Omega f g \,d\mu
\]
are two equivalent standard norms considered in $L^\Phi$. In fact,  $\|f\|_\Phi \le \|f\|_\Phi^0 \le 2 \|f\|_\Phi$ for any $f\in L^\Phi$.  In particular case when  $\Phi$ does not depend on the parameter, that is $\Phi(x,t) = \varphi(t)$, for a.a. $x\in\Omega$, $t\ge 0$, where $\varphi$ is an Orlicz function, then $L^\Phi$ is an Orlicz space.  The Musielak-Orlicz space $L^\Phi$ with either norm is a Banach function lattice satisfying the Fatou property.  For any $MO$ function $\Phi$, $(L^\Phi)_a = (L^\Phi)_b$. If we do not mention otherwise, we will always consider the $MO$ space $L^\Phi$ equipped with the Luxemburg norm. Extensive information about Musielak-Orlicz spaces one can find in \cite{CUF, DHHR, HH, H1983, HK, KamKub, Mus}.

Given an $eMO$ function $\Phi$, by $\Phi^*$ denote the {\it complementary function} to $\Phi$, that is 
\[
\Phi^*(x,t) = \sup_{s\ge0} \{st - \Phi(x,s)\}, \ \ \ a.a.\  x\in\Omega, \ t\ge0.
\]
The reason that we also consider $eMO$ functions in this paper is that even if a $MO$ function $\Phi$ has finite values, its complementary function $\Phi^*$ may achieve infinite values.  The function $\Phi(x,t) = t$, $x\in \Omega$, $t\ge 0$, is the simplest  example of such functions.
It is well known and not difficult to show  that  $\Phi^*$ is  $eMO$ function that is  a Musielak-Orlicz function with extended values  and $\Phi^{**} =\Phi$. 

 In view of the simple observation that $I_\Phi(f) \le 1$ if and only if $\|f\|_\Phi \le 1 $,
 the following H\"older inequalities are satisfied  for any $f\in L^\Phi$, $g\in L^{\Phi^*}$, 
\[
\left|\int_\Omega f\,g\,d\mu\right| \le \|f\|_\Phi^0\, \|g\|_{\Phi^*}, \ \ \ \ \left|\int_\Omega f\,g\,d\mu\right| \le \|f\|_\Phi\, \|g\|_{\Phi^*}^0. 
\]
We say that $MO$ function $\Phi$ satisfies {\it condition} $\Delta_2$  if there exist $K>0$ and a  non-negative integrable function $h$ on $\Omega$, that is $h\in L^1= L^1(\Omega)$ such that 
\[
\Phi(x, 2t) \le K \Phi(x,t) + h(x), \ \ \ a.a.\  x\in\Omega,\ \  t\ge 0.
\]
  Let $\Phi$ be an Orlicz function, that is $\Phi(x,t)=\varphi(t)$ for a.a. $x\in\Omega$. One can show that when $\mu(\Omega)<\infty$ and $\mu$ is non-atomic then  $\Phi$ satisfies $\Delta_2$ if and only if for some $k>0$ and $u_0\geq 0$,  
\begin{equation}\label{ineq:deltatwo}
\varphi(2u)\leq k\varphi(u)\ \ \ \ \text{for all}\ \ \  u\geq u_0.   \tag{$\Delta_2^\infty$}
\end{equation}

The growth condition $\Delta_2$ of $\Phi$ plays an important role in the theory of $MO$ spaces. 
 Recall now some results in $MO$ spaces which we will need later.

\begin{theorem}\cite[Theorem 7.6,Theorem 8.14]{Mus}\cite[Proposition 3.10, Theorem 3.13]{BS}\label{th:ordconMO}

  Let $\Phi$ be a $MO$ function. The following properties are equivalent.
\begin{itemize}
\item[{\rm(i)}]  The space $L^\Phi$ is order continuous that is  $L^\Phi = (L^\Phi)_a =  (L^\Phi)_b$. 
\item[{\rm(ii)}]  The modular convergence $I_\Phi(u)\to 0$ is equivalent to norm convergence $\|u\|_\Phi \to 0$.
\item[{\rm(iii)}] $\Phi$ satisfies $\Delta_2$.
\end{itemize}
\end{theorem}

\begin{theorem}\cite[Theorem 7.10]{Mus}\label{th:sepMO}
  The $MO$ space $L^\Phi(\Omega)$ is separable if and only if the measure $\mu$ is separable and $\Phi$ satisfies $\Delta_2$. 

\end{theorem}

\begin{theorem} \label{th:funcMO}
Let $\Phi$ be $MO$ function.
\begin{itemize}
\item[{\rm (i)}] \cite[Theorem A4]{KamKub} \cite{Mus} \label{th:Kothe}
 $(L^\Phi, \|\cdot\|_\Phi)' = (L^{\Phi^*}, \|\cdot\|_{\Phi^*}^0)$ and  $(L^\Phi, \|\cdot\|_\Phi^0)' = (L^{\Phi^*}, \|\cdot\|_{\Phi^*})$. 
\item[{\rm (ii)}] \cite{Z}
\[
(L^\Phi, \|\cdot\|_\Phi)^* \simeq (L^{\Phi^*}, \|\cdot\|^0_{\Phi^*}) \oplus (L^\Phi)_s,
\]
where $ (L^\Phi)_s = ((L^\Phi)_a)^\perp.$
\item[{\rm (iii)}] \cite{Mus}  $L^\Phi$ is reflexive if and only if both $\Phi$ and $\Phi^*$ satisfy condition $\Delta_2$. 
\end{itemize}
\end{theorem}

  Given $MO$ functions $\Phi_i$, $i=1,2$, the symbol $\Phi_2 \prec \Phi_1$ denotes that  there exist a constant $K>0$ and a non-negative function $h\in L^1$ such that
\[
\Phi_2(x,Kt)\le \Phi_1(x,t) + h(t), \ \ \ \ a.a.\ x\in\Omega, \ t\ge 0.
\]
We say that $\Phi_1$ and $\Phi_2$ are {\it equivalent} if $\Phi_2 \prec \Phi_1$ and $\Phi_1 \prec \Phi_2$.  The equivalence of two $MO$ functions preserves condition $\Delta_2$.

\begin{theorem}\cite{Mus}
\label{th:equivMO}
Given $MO$ functions $\Phi_i$, $i=1,2$, $L^{\Phi_1} \subset L^{\Phi_2}$ if and only if $\Phi_2 \prec \Phi_1$.
The embedding of the spaces $L^{\Phi_1} \subset L^{\Phi_2}$ is automatically  bounded. Consequently $L^{\Phi_1} = L^{\Phi_2}$ as sets with equivalent norms if and only if $\Phi_1$ is equivalent to $\Phi_2$.
\end{theorem}
\begin{proof}
The proof of the first part can be found in \cite{Mus}. The automatic boundedness of the embedding among $MO$ spaces follows from  \cite[Proposition 2.10, p. 13]{BS}, in view of the Fatou property of $MO$ spaces.

\end{proof}

 For a $eMO$ function $\Phi: \Omega \times \R_+ \to [0,\infty]$ define  the {\it generalized inverse} of $\Phi$ by the formula
 \[
 \Phi^{-1}(x,t)=\inf\{u:\Phi(x,u)\geq t\}, \ \ \ \ a.a. \ x\in\Omega, \ \ t \ge 0.
  \]
  Clearly we have
$ \Phi(x,\Phi^{-1}(x,t)) \le t$ for  a.a. $ x\in\Omega$ and $t \ge 0$.

\begin{proposition} Let $\Phi$ be $eMO$ function. Then 
\begin{itemize}
\item[{(i)}] The {\it Young inequality} 
\[
st \le \Phi(x,s) + \Phi^*(x,t), \ \ \ \ a.a. \ x\in\Omega \ \ s,t \ge 0,
\]
\item[{(ii)}] For  a.a. $x\in \Omega$,  $t\ge 0$, 
\begin{equation}\label{eq:inverseineq}
 t \le (\Phi^*)^{-1}(x,t) \Phi^{-1}(x,t) \le 2t.
\end{equation}
\end{itemize}
\end{proposition}
\begin{proof} (i)  It is a direct consequence of the definition of the complementary function.

(ii) This is well known for Orlicz $N$-functions (e.g. \cite{KR}). For the sake of completeness we provide a short proof here.  For the right inequality, take  $x\in\Omega$ and $t\geq 0$, and let $u=\Phi^{-1}(x,t)$ and $v=(\Phi^*)^{-1}(x,t)$. Then 
\[ 
(\Phi^*)^{-1}(x,t) \Phi^{-1}(x,t)=uv\leq \Phi(x,u)+\Phi^*(x,v)\leq 2t
\]
by the Young inequality.
For the left inequality, recall that 
\[
\Phi(x,t)=\int\limits_0^t\Phi'(x,s)ds, \ \ \ \Phi^*(x,t)=\int\limits_0^t (\Phi')^{-1}(x,s)ds, \ \ \ \ a.a.\ x\in \Omega, \ \  t\ge 0, 
\]
where $\Phi'(x,t)$ is the right derivative of $\Phi(x,t)$  with respect to $t\ge 0$ for a.a. $x\in \Omega$, and $(\Phi')^{-1}$ is a generalized inverse as defined above for $\Phi$. For any $x\in\Omega$ and $u>0$ such that $\Phi(x,u)=\int\limits_{0}^u\Phi'(x,s)ds\geq t$ we have
 \[
 \frac{t}{u}\leq\frac{\Phi(x,u)}{u}=\frac{1}{u}\int\limits_0^u \Phi'(x,s)ds\leq\frac{\Phi'(x,u)u}{u}=\Phi'(x,u).
 \]
It follows in view of $(\Phi')^{-1}(x,\Phi'(x,u))\le u$,
 \[
 \Phi^*\left(x,\frac{t}{u}\right)=\int\limits_0^{\frac{t}{u}}(\Phi')^{-1}(x,s)ds\leq (\Phi')^{-1}\left(x,\frac{t}{u}\right)\frac{t}{u}\leq (\Phi')^{-1}(x,\Phi'(x,u))\frac{t}{u}\leq t. 
 \]
Applying now $(\Phi^*)^{-1}$ to both sides of the above inequality,  for  any $u, t\ge 0$ with $\Phi(x,u) \ge t$, we get
 \[
 u(\Phi^{*})^{-1}(x,t)\geq t.
 \]
 Hence for a.a. $x\in \Omega$, $t\ge 0$,
 \[
 \Phi^{-1}(x,t)(\Phi^*)^{-1}(x,t)=\inf\{u:\Phi(x,u)\geq t\}(\Phi^*)^{-1}(x,t)\geq t,
 \]
which is the left inequality of (\ref{eq:inverseineq}).

\end{proof}


\begin{lemma}\label{delta2compl}
Let $\Phi'(x,t)$ be the right derivative of a $MO$ function $\Phi(x,t)$ for a.a. $x\in \Omega$ with respect to $t \ge 0$.  If   exists a constant $k > 1$ such that
\begin{equation}\label{eq:UC1}
 \Phi'(x,2 t) \ge k \Phi'(x, t), \ \ \ \  a.a. \ x\in\Omega, \ \ \ t\ge 0,
 \end{equation}
then $\Phi^*$ satisfies condition $\Delta_2$. 
\end{lemma}
\begin{proof}
Integrating the inequality (\ref{eq:UC1}), we get for $s\ge 0$,
\begin{align*}
\int_0^s \Phi'(x,2t)\, dt &= \frac{1}{2}\int_0^{2s} \Phi'(x,u) \, du \\
&= \frac{1}{2} \Phi(x,2s )\ge k\int_0^s \Phi'(x,t)\, dt = k \Phi(x,s).
\end{align*}
Hence $\Phi(x, 2s) 
\ge 2k \Phi(x,s)$. 
 It follows $\Phi(x, \frac{s}{2}) \le \frac{1}{2k} \Phi(x,s)$, and thus
\begin{align*}
\Phi^*(x,2s) =\sup_{u\ge0}\left\{us - \Phi\left(x,\frac{u}{2}\right)\right\} \ge \sup_{u\ge 0} \left\{us - \frac{1}{2k} \Phi(x,u)\right\} = \frac{1}{2k}\Phi^*(x,2k s).
\end{align*}
Hence for a.a. $x\in \Omega$, $s\ge 0$,
\[
\Phi^*(x,ks) \le 2k \Phi^*(x,s),
\]
which implies $\Delta_2$ condition of $\Phi^*$ by the assumption $k > 1$.

\end{proof}

  The important class of Musielak-Orlicz spaces are   {\it Nakano spaces},  recently also called {\it variable exponent Lebesgue spaces}. Let $1\le p(x) < \infty$ for a.a.  $x\in\Omega$, be a measurable function and let
\begin{equation}\label{eq:lebvar}
\Phi(x,t) = \frac{t^{p(x)}}{p(x)}, \ \ \  a.a. \  x\in\Omega,\ t\ge 0. 
\end{equation}
Then $L^{p(\cdot)} = L^\Phi$  is a variable exponent Lebesgue space. For $1<p(x)<\infty$ a.e. in $\Omega$, $t\ge 0$,
\[
\Phi^*(x,t) = \frac{t^{q(x)}}{q(x)},
\]
where $\frac{1}{p(x)}  + \frac{1}{q(x)} =1$. Let further
\[
p^+ = \esssup_{x\in\Omega} p(x) \ \ \ \text{and} \ \ \ \ p^- = \essinf_{x\in\Omega} p(x).
\]

In the literature there is another version of the Nakano spaces given by the $MO$ function $t^{p(x)}$ for $1\le p(x) < \infty$  a.a. $x\in \Omega$, $t\ge 0$. The spaces defined in either way are equal as sets with  equivalent norms. In order to avoid confusion we will consider here only the spaces given by formula (\ref{eq:lebvar}).

\begin{theorem}\label{th:deltalpx}
The variable exponent $MO$ function $\Phi(x,t) = t^{p(x)}/p(x)$, $1\le p(x) <\infty$, for a.a. $x\in\Omega$, $t\ge 0$,  satisfies $\Delta_2$ if and only if $p^+ < \infty$. Its complement function $\Phi^*$ satisfies $\Delta_2$ if and only if $p^- > 1$.
\end{theorem}
 \begin{proof}
 Assume that   $\Phi(x,t) = t^{p(x)}/p(x)$ satisfies $\Delta_2$. There exists $C\geq2$ and an positive, integrable function $h$ such that for a.a. $x\in\Omega$ and $t\geq0$ we have
 \[
 2^{p(x)}\frac{t^{p(x)}}{p(x)}=\Phi(x,2t)\leq C\Phi(x,t)+h(x)=C\frac{t^{p(x)}}{p(x)}+ h(x).
 \]
 Hence, 
 \[(2^{p(x)}-C)\frac{t^{p(x)}}{p(x)}\leq h(x).\]
 If $p(x)=1$ for a.e. $x\in \Omega$, then clearly $\Phi$ satisfies $\Delta_2$ condition. Now if $p(x)\neq 1$ for a.a. $x\in\Omega$, then for a.a. $x\in\Omega$ it follows that $\sup\limits_{t>0}\frac{t^{p(x)}}{p(x)}=\infty$. Therefore we conclude that 
 \[
 2^{p(x)}\leq C, \ \ \ \ a.a. \ x\in \Omega.
 \]
It follows   that for a.a. $x\in\Omega$ we have $p(x)\leq\log_2C$ and so $p^+\leq\log_2C<\infty$. If on the other hand $p^+<\infty$, then for a.a. $x\in\Omega$ and any $t\geq0$ we have
 \[
 \Phi(x,2t)=2^{p(x)}\frac{t^{p(x)}}{p(x)}\leq 2^{p^+}\frac{t^{p(x)}}{p(x)}=2^{p+}\Phi(x,t), 
 \]
 that is $\Phi$ satisfies $\Delta_2$.
 
   As for the second assertion, notice that $\Phi^*(x,t)=t^{q(x)}/q(x)$, where $q(x)=p(x)/(p(x)-1)$. By the first part of the proof $\Phi^*(x,t)$ satisfies $\Delta_2$ if and only if $q^+<\infty$.  By the fact that $u/(u-1)$ is a decreasing function on $(1,\infty)$ and $\lim\limits_{u\to1^+} u/(u-1)=\infty$,  we have
 \[
 q(x)=\frac{p(x)}{p(x)-1}\leq\frac{p^-}{p^--1}, \ \ \ \  a.a. \ x\in\Omega.
 \]
 Therefore we conclude that $q^+<\infty$ if and only if $p^->1$.

\end{proof}

Another important class of $MO$ functions is a class of double phase functionals consisting of $\Phi: \Omega\times \mathbb{R}_+ \to \mathbb{R}_+$ such that 
\[
\Phi(x,t) = t^{p(x)} + a(x) t^{r(x)},
\] 
where $a,p,r$ are measurable functions on $\Omega$,  $a(x) \ge 0$ a.e. and $1\le p(x) \le r(x) <\infty$ a.e.. Denote by $r^+$ and $r^-$ analogously as $p^+$ and $p^-$ respectively.

\begin{theorem}\label{th:deltadobphase}
Let 
\[
\Phi(x,t) = t^{p(x)} + a(x) t^{r(x)},
\]
where $a,p,r$ are measurable functions on $\Omega$,  $a(x) \ge 0$ a.e. and $1\le p(x) \le r(x) <\infty$ a.e..

{\rm{(i)}} If $r^+ < \infty$ then $\Phi$ satisfies $\Delta_2$.

{\rm{(ii)}} If $p^- > 1$ then $\Phi^*$ satisfies condition $\Delta_2$.

{\rm{(iii)}} If in addition we assume  that $p(x) < r(x)$ for a.a. $x\in {\Omega_1}$,
where $\Omega_1 = \supp a$,
then  $\Phi$ satisfies condition $\Delta_2$ if and only if $p^+< \infty$ and $r^+|_{\Omega_1}=\esssup_{x\in \Omega_1} r(x) < \infty$. 
\end{theorem}

\begin{proof} 
If $\mu(\Omega_1) = 0$ that is  $a(x) = 0$ for a.a. $x\in \Omega$, then the space $L^\Phi$ is reduced to the Lebesgue variable space $L^{p(\cdot)}$, and the conclusion follows from Theorem \ref{th:deltalpx}. Thus  assume $\mu(\Omega_1) > 0$.

(i) If $p^+<\infty$ and $r^+|_{\Omega_1} < \infty$ then
\[
\Phi(x,2t) = 2^{p(x)} t^{p(x)} + 2^{r(x)} a(x) t^{r(x)} \le 2^r(t^{p(x)} + a(x) t^{r(x)}) = 2^r \Phi(x,t),
\]
where $r = \max (p^+, r^+|_{\Omega_1})$, and so $\Phi$ satisfies condition $\Delta_2$.

(ii) 
Let $p^-> 1$. Then for   a.a. $x\in\Omega$, $t\ge 0$,
\begin{align*}
 \Phi'(x,2 t)  &= p(x) 2^{p(x)-1} t^{p(x)-1}  + a(x) r(x) 2^{r(x) -1} t^{r(x) - 1}\\
 &\ge 2^{p^- - 1} p(x) t^{p(x) -1}  + a(x) r(x) 2^{p^- -1} t^{r(x) - 1} =  2^{p^- - 1}\Phi'(x, t),
 \end{align*}
 where  $2^{p^- - 1}>1$. We conclude by Lemma \ref{delta2compl}.

(iii) Now let $\Phi$ satisfy $\Delta_2$ and $p^+ = \infty$ or $r^+|_{\Omega_1} = \infty$. Then by definition of $\Delta_2$, there exists $c>0$ such that the function
\[
h_c(x) = \sup_{t\ge 0} (\Phi(x,2t) - c \Phi(x,t))
\]
belongs to $L^1$. We shall consider three cases.

Case $1^0$. Let $p^+|_{\Omega \setminus \Omega_1} =  \esssup_{x\in \Omega\setminus \Omega_1} p(x) =\infty$ if $\mu(\Omega\setminus \Omega_1) > 0$. Then $\Phi(x,t) = t^{p(x)}$ for a.a. $x\in \Omega\setminus\Omega_1$, and by Theorem \ref{th:deltalpx} it is a contradiction because $\Phi$ can not satisfy condition $\Delta_2$.

  Case $2^0$. Let $p^+|_{\Omega_1} = \infty$. Define
\[
A_n = \{x\in \Omega_1: p(x) > n\}, \ \ \ n\in\N.
\]
By assumptions for every $n\in\N$, $\mu(A_n) > 0$, and for a.a. $x\in A_n$,
\[
r(x) > p(x) > n.
\]
Choose $n\in\N$ such that $2^n > c$. Then for $x\in A_n$, $t\ge 0$,
\begin{align*}
\Phi(x,2t) - c \Phi(x,t) &\ge 2^n t^{p(x)} + a(x) 2^n t^{r(x)} - c t^{p(x)} - c  a(x) t^{r(x)} \\
&= (2^n-c)(t^{p(x)} + a(x) t^{r(x)}) = (2^n -c) \Phi(x,t).
\end{align*}
Hence $h_c(x) =\infty$ for every $x\in A_n$ and $c < 2^n$. Therefore  $h_c$ is not integrable over $\Omega$, and by monotonicity of $h_c$ with respect to $c$, $h_c\notin L^1$ for every $c>0$. Consequently, $\Phi$ does not satisfy $\Delta_2$.\\

  Case $3^0$. Let $p^+ < \infty$ and $r^+|_{\Omega_1} = \infty$.
 We can assume that $c > 2^{p+}$. Thus 
\[
2^{p(x)} - c < 0, \ \ \ \ a.a.\  x\in\Omega.
\]
Let $x\in\Omega_1$. Then $a(x) > 0$ and $p(x) - r(x) < 0$ by assumptions. Therefore there exists  $T_x>2$ such that for all $t>T_x$,
\[
t^{p(x) - r(x)}  < a(x).
\]
For a.a. $y\in \Omega$, we have $\Phi(y,2t) - c\Phi(y,t) = t^{r(y)} [(2^{p(y)} -c) t^{p(y) - r(y)} + (2^{r(y)} - c) a(y)]$. Hence for a fixed $t>T_x$,
\begin{align*}
\Phi(x,2t) - c\Phi(x,t) &= t^{r(x)} [(2^{p(x)} - c) \,t^{p(x)-r(x)} + (2^{r(x)} - c) \, a(x) ]\\
&\ge t^{r(x)} [(2^{p(x)} - c) \,a(x) + (2^{r(x)} - c) \, a(x) ]\\
&= t^{r(x)} a(x) ( 2^{p(x)} + 2^{r(x)} - 2c ) > t^{r(x)} a(x) ( 2^{r(x)} - 2c).
\end{align*}
Let for $n\in\N$,
\[
B_n = \{ x \in \Omega_1: \ r(x) > n\}.
\]
Since $r^+|_{\Omega_1} = \infty$, there is $N$ such that for all $n > N$,  $\mu(B_n) > 0$ and $2^n > 2c$. Therefore for $t\ge 0$, $x\in B_n$ and $n> N$,
\[
 t^{r(x)} a(x) ( 2^{r(x)} - 2c) \ge  t^{r(x)} a(x) ( 2^n - 2c).
 \]
 Thus for any $ x\in B_n$ and $t_x>T_x$, 
 \[
\sup_{t\ge 0}\{\Phi(x,2t) - c\Phi(x,t)\} \ge  \sup_{t\ge 0} \, [t^{r(x)} a(x)\, ( 2^{r(x)} - 2c) ]\ge  t_x^{r(x)} a(x) \,( 2^n - 2c)   .
\]
Since $t_x$ can be taken arbitrary big it follows, for $n> N$ and for a.a. $x\in B_n $,
\[
h_c(t) \ge \sup_{t\ge 0} \, [t^{r(x)} a(x)\, ( 2^{r(x)} - 2c) ] = \infty. 
\]
Hence $h_c(x) =\infty$ for every $x\in B_n$. Therefore, if $c<2^n$, then $h_c$ is not integrable, but $n$ can be chosen arbitrary large, so for every $c>0$, $h_c\notin L^1$.

\end{proof}

In applications there are  often used double phase functionals  where the functions $p(x)$ and $r(x)$ are constants. Consequently in view of Theorem \ref{th:deltadobphase} we get the corollary.

\begin{corollary}  Let $\Phi$ be a double phase functional of the following form
\[
\Phi(x,t) = t^p + a(x) t^r, \ \ \ \   a.a. \ x\in \Omega, \ \  t\ge 0,
\]
where $1\le p \le  r <\infty$ and  $a(x) \ge 0$ for a.a. $x\in \Omega$. Then $\Phi$ satisfies $\Delta_2$, and if $p>1$ then $\Phi^*$ fulfills $\Delta_2$.
\end{corollary}

  Let $\Omega = (\alpha, \beta)$, $-\infty < \alpha \le \beta < \infty$ with the Lebesgue measure $\mu$. 
By $L^1_{loc} = L^1_{loc} (\Omega)$ denote the set of all {\it locally integrable} functions on $\Omega$, that is all functions $f\in L^0$ such that $\int_K |f| \, d\mu<\infty$ for every compact set $K\subset \Omega$. By $C_c^\infty(\Omega)$ we denote the space of all smooth  complex valued functions on $\Omega$ with compact supports. By a smooth function we mean a function which has all derivatives.    Recall, a function $f\in L^1_{loc}$ is {\it weakly  differentable} if there exists a function $f'\in L^1_{loc}$ such that for every $u\in C_c^\infty(\Omega)$ we have
\[
\int\limits_{\alpha}^{\beta}f(x)u'(x)dx=-\int\limits_{\alpha}^{\beta}f'(x)u(x)dx.
\]
Then the function $f'$ is called the {\it  weak derivative} of $f$. If $f:\Omega\to \R$ is absolutely continuous on every compact subinterval of $\Omega$, then $f$ is weakly differentiable and $f'$ coincides with the classical derivative of $f$ a.e. \cite[Theorem 7.16]{Giovanni Leoni}.

Given a $MO$ function $\Phi$, the Musielak-Orlicz-Sobolev space  ($MOS$ space) $W^{1,\Phi} =W^{1,\Phi} (\Omega)$  consists of all $f\in L^1_{loc}(\Omega)$ such that their weak derivative $f'$ exists   and
\begin{equation}\label{eq:norm}
\|f\|_{1,\Phi} = \|f\|_\Phi + \|f'\|_\Phi < \infty.
\end{equation}
The space $W^{1,\Phi}$ is often called  a generalized Sobolev space, or if we know the context just a Sobolev space.  If $\Phi$ is an Orlicz function $\varphi$ then $W^{1,\varphi}$ is an Orlicz-Sobolev space. 
  In the case of variable exponent $MO$ function,  the variable exponent Sobolev space is denoted by  $W^{1,p(\cdot)}$.
 The space $W^{1,\Phi}$ equipped with the norm $\|\cdot\|_{1,\Phi}$ is a complete space \cite[Theorem 6. 1.4]{HH}.  In the case of $W^{1,p(\cdot)}$ see also (\cite[Theorem 6.6]{CUF}, \cite[Theorem 8.1.6]{DHHR}). 
 
   Anytime further we will use Sobolev spaces $W^{1,\Phi}$, they will always be defined on the finite interval $(\alpha, \beta)$ equipped with the Lebesgue measure.

\section{Integral Operators in $L^\Phi$}\label{seciton 4.2}

The integral operators in Orlicz spaces have been studied  among others in \cite{KR, PS}. Here we will study the integral operators in Musielak-Orlicz spaces. Let  $(\Omega,\Sigma, \mu )$ be a $\sigma$-finite measure space. For $MO$ functions $\Phi_1, \Phi_2$ define the function
\begin{equation}\label{eq:11}
\phi((x,y),t) = \Phi_2(x,\Phi^*_1(y,t)), \ \ \ a.a. \ (x,y) \in \Omega\times \Omega, \ t\ge0.
\end{equation} 
Note first that $\Phi_1^*$  can achieve infinite values and so it is $eMO$ function.  We will adopt the convention that if $\Phi_1^*(y,t) = \infty$ then $\phi((x,y),t) = \Phi_2(x,\infty) = \infty$. Therefore $\phi :  (\Omega\times \Omega) \times \R_+ \to [0,\infty]$ is a $eMO$. Denote by $L^\phi$ the Musielak-Orlicz space as a subspace of $L^0(\Omega\times \Omega)$. By $\|\cdot\|_\phi$ and $\|\cdot\|_\phi^0$ we mean the Luxemburg and Orlicz norm on $L^\phi$, respectively.

\begin{lemma} \label{lem:11} Let $\Phi_i, i=1,2$, be $MO$ functions on $\Omega$, where $\mu(\Omega) < \infty$. Assume $\int_\Omega \Phi_2(x,b) d\mu(x) < \infty$ for some $b>0$. Let $\psi: (\Omega\times \Omega) \times \R_+\to [0,\infty]$ be such that
\[
\psi((x,y),t) = \Phi_2(x, \Phi^*_1(y,t)), \ \ \ a.a. \ (x,y)\in \Omega\times\Omega, \ t\ge 0,
\]
and let $\phi = \psi^*$.
Then there exists $l>0$ such that  whenever $u\in L^{\Phi_1}$, $v\in L^{\Phi_2}$ and
\[
w(x,y) = u(y) v(x), \ \ \ a.a.\ \ x,y\in\Omega,
\]
then 
\[
\|w\|_\phi^0 \le l \,\|u\|_{\Phi_1} \|y\|_{\Phi^*_2}.
\]
\end{lemma}

\begin{proof}

 By Young's inequality applied to $\Phi_1$ and $\Phi_1^*$ and the measure $\frac{|v(x)|}{\|v\|_{\Phi_2^*}} d\mu(x)$, for $b>0$ from the assumptions we get
\begin{align*}
&\left|\int_\Omega\int_\Omega bw(x,y) g(x,y)\,d\mu(x)d\mu(y)\right|
\le \|u\|_{\Phi_1}\|v\|_{\Phi_2^*}\, b \int_\Omega \int_\Omega |g(x,y)| \frac{|u(y)||v(x)|}{\|u\|_{\Phi_1}\|v\|_{\Phi_2^*}}\,d\mu(x)d\mu(y)\\
&\le \|u\|_{\Phi_1}\|v\|_{\Phi_2^*}\, b\int_\Omega \left[\int_\Omega \Phi_1^*(y,|g(x,y)|) \frac{|v(x)|}{\|v\|_{\Phi_2^*} }\,d\mu(x) + \int_\Omega \Phi_1\left(y, \frac{|u(y)|}{\|u\|_{\Phi_1}}\right) \frac{|v(x)|}{\|v\|_{\Phi_2^*}}\,d\mu(x)\right]\,d\mu(y)\\
& = \|u\|_{\Phi_1}\|v\|_{\Phi_2^*}\,(A + B).
\end{align*}
Applying now  to term $A$, Young's inequality to $\Phi_2$ and $\Phi_2^*$ and the obvious fact   $I_{\Phi_2^*} (v/\|v\|_{\Phi_2^*} )\le 1$, we get
\begin{align*}
A &= b\int_\Omega \left[\int_\Omega \Phi_1^*(y,|g(x,y)|) \frac{|v(x)|}{\|v\|_{\Phi_2^*} }\,d\mu(x)\right]\, d\mu(y)\\
&\le b \int_\Omega \left[\int_\Omega \Phi_2(x, \Phi_1^*(y,|g(x,y)|))\, d\mu(x) + \int_\Omega \Phi_2^*\left(x, \frac{|v(x)|}{\|v\|_{\Phi_2^*}}\right)\,d\mu(x)\right]\,d\mu(y)\\
&\le b\int_\Omega\int_\Omega \psi((x,y), |g(x,y)|) \,d\mu(x)d\mu(y) + b\mu(\Omega).
\end{align*}
 By  Fubini's theorem and $I_{\Phi_1} (u/\|u\|_{\Phi_1} )\le 1$, and Young's inequality applied  to $\Phi_2$,
\begin{align*}
B &= \int_\Omega \left[\int_\Omega \Phi_1\left(y, \frac{|u(y)|}{\|u\|_{\Phi_1}}\right)\,d\mu(y)\right]\, b \frac{|v(x)|}{\|v\|_{\Phi_2^*}}\,d\mu(x)\le \int_\Omega b \frac{|v(x)|}{\|v\|_{\Phi_2^*}}\,d\mu(x)\\
&\le \int_\Omega \Phi_2^*\left(x, \frac{|v(x)|}{\|v\|_{\Phi_2^*}}\right) \,d\mu(x)+ \int_\Omega \Phi_2(x,b)\,d\mu(x) \le  1 + \int_\Omega \Phi_2(x,b)\,d\mu(x).
\end{align*}

By the above,
\begin{align*}
&\left|\int_\Omega\int_\Omega bw(x,y) g(x,y)\,d\mu(x)d\mu(y)\right|\le \|u\|_{\Phi_1}\|v\|_{\Phi_2^*}\,(A + B)\\
&\le  \|u\|_{\Phi_1}\|v\|_{\Phi_2^*} \left( b\int_\Omega\int_\Omega \psi((x,y), |g(x,y)|) \,d\mu(x)d\mu(y) + b\mu(\Omega) + 1 + \int_\Omega \Phi_2(x,b)\,d\mu(x)\right)\\\
&= \|u\|_{\Phi_1}\|v\|_{\Phi_2^*} \left(b I_ \psi (g) + b\mu(\Omega)+ 1 + \int_\Omega \Phi_2(x,b)\,d\mu(x)\right).
\end{align*}

Finally,
\[
\|w\|_\phi^0 = \sup_{I_\psi(g)\le 1}\left|\int_\Omega\int_\Omega w(x,y) g(x,y)\,d\mu(x)d\mu(y)\right| \le l\,\|u\|_{\Phi_1} \|v\|_{\Phi_2^*},
\]
where $l= b + b\mu(\Omega) + 1+ \int_\Omega \Phi_2(x,b)\,d\mu(x)<\infty$ by assumption.

\end{proof}

\begin{theorem}\label{th:11}
Let $\Phi_i$, $i=1,2$, be $MO$ functions on $\Omega$. Assume $\phi:(\Omega\times\Omega)\times \R_+\to [0,\infty]$ is a $eMO$ function and there exists $l>0$ such that  if $u\in L^{\Phi_1}$, $v\in L^{\Phi_2}$ and $w(x,y) = u(y) v(x)\in L^\phi$ then 
\[
\|w\|_\phi^0 \le l \|u\|_{\Phi_1} \|v\|_{\Phi_2^*}.
\]
 Let $A$ be an integral operator for $u\in L^0$,
\[
Au(x) = \int_\Omega k(x,y) u(y)\,d\mu(y), \ \ \ a.a. \ x\in\Omega,
\]
where the kernel $k(x,y)\in L^{\phi^*}$. Then $A: L^{\Phi_1}\to L^{\Phi_2}$ is bounded.
\end{theorem}

\begin{proof}
By H\"older's inequality, for $u\in L^{\Phi_1}$, $v\in L^{\Phi_2^*}$,
\begin{align*}
\left|\int_\Omega Au(x) v(x)\,d\mu(x)\right| &= \left|\int_\Omega\int_\Omega k(x,y) u(y)v(x)\,d\mu(x)d\mu(y)\right|\\ 
&= \left|\int_\Omega\int_\Omega k(x,y) w(x,y)\,d\mu(x)d\mu(y)\right|\\ 
&\le \|k\|_{\phi^*} \|w\|_\phi^0 \le l \|k\|_{\phi^*} \|u\|_{\Phi_1} \|v\|_{\Phi_2^*}.
\end{align*}
Recall that $I_{\Phi_2^*}(v) \le 1$ if and only if $\|v\|_{\Phi_2^*} \le 1$. Consequently,
\[
\|Au\|^0_{\Phi_2} = \sup_{I_{\Phi_2^*}(v) \le 1 }\left|\int_\Omega Au(x)v(x)\,d\mu(x)\right| \le  l \|k\|_{\phi^*} \|u\|_{\Phi_1},
\]
and thus $A$ is bounded from $L^{\Phi_1}$ to $L^{\Phi_2}$. 

\end{proof} 

\begin{corollary}\label{cor:11}
Let $\Phi_i, \, i=1,2$, be $MO$ functions on $\Omega$ and $\mu(\Omega)<\infty$. Assume there exists $b>0$ such that $\int_\Omega \Phi_2(x,b)\,d\mu(x) < \infty$. Let $\psi: (\Omega\times \Omega) \times \R_+\to [0,\infty]$ be such that
\[
\psi((x,y),t) = \Phi_2(x, \Phi^*_1(y,t)), \ \ \ \ \ a.a. \ x,y\in\Omega, \ \  t\ge 0.
\]
If $k(x,y)\in L^\psi$, then the operator $Au(x) = \int_\Omega k(x,y) u(y)\,d\mu(y)$ is bounded from $L^{\Phi_1}$ to $L^{\Phi_2}$.
\end{corollary}
 
 \begin{proof} By Lemma \ref{lem:11}, there is $l>0$ such that for any $u\in L^{\Phi_1}$, $v\in L^{\Phi^*_2}$, and $w(x,y) = u(y) v(x)$ we have $\|w\|^0_\phi \le l \|u\|_{\Phi_1} \|v\|_{\Phi_2^*}$, where $\phi = \psi^*$. Consequently, the assumptions of Theorem \ref{th:11} are satisfied and $A$ is bounded.

\end{proof}

The next result  follows immediately from Corollary \ref{cor:11}.

\begin{corollary}\label{cor:12}
Let $\Phi$ be a $MO$ function on $\Omega$ with $\mu(\Omega)<\infty$ and such that $\int_\Omega \Phi(x,b)\,d\mu(x) < \infty$ for some $b>0$. Let 
\[
\psi((x,y),t) = \Phi(x,\Phi^*(y,t)), \ \ \ x,y\in\Omega,\ t\ge 0.
\]
If $k(x,y) \in L^\psi$ then $Au(x) = \int_\Omega k(x,y) u(y)\,d\mu(y)$ is bounded from $L^{\Phi}$ to $L^{\Phi}$.
\end{corollary}

Now we wish to formulate a condition for the Voltera operator to be bounded from $L^\Phi(\alpha, \beta)$ to itself for $-\infty < \alpha < \beta < \infty$. It requires some preparations.

\begin{lemma}\label{lem:111}
Given a $MO$ function $\Phi$ on $\Omega$, we have that $\essinf_{x\in\Omega}\Phi(x,a) >0$ for some $a>0$ if and only if there exists $c>0$ with $\esssup_{x\in\Omega} \Phi^*(x,c) <\infty$.
\end{lemma}

\begin{proof}
If   $\essinf_{x\in\Omega}\Phi(x,a) >0$ for some $a>0$, then $\Phi(x,a) \ge M >0$ for a.a. $x\in \Omega$ and some $M$. Hence $a \ge \Phi^{-1}(x,M)$. Thus in view of (\ref{eq:inverseineq}),
$(\Phi^*)^{-1}(x,M) \ge \frac{M}{\Phi^{-1}(x,M)}\ge \frac{M}{a}$. Therefore 
$\infty >M \ge \esssup_{x\in\Omega}\Phi^*(x, c)$ with $c=\frac{M}{a}$. In the opposite direction the proof goes in a similar way.

\end{proof}

\begin{definition}\label{V condtion}
Let $\Omega= (\alpha, \beta)$, $-\infty< \alpha < \beta < \infty$,  equipped with the Lebesgue measure. We say that $MO$ function $\Phi$ satisfies condition {\rm (V)} if for some $a,b>0$,
\begin{equation}\tag{V}
 \int_\alpha^\beta \Phi(x,b)\,dx < \infty\ \ \text{and}\ \  \ \essinf_{x\in\Omega} \Phi(x,a) >0.
 \end{equation} 
\end{definition}

\begin{theorem}\label{th:12} Let $\Omega= (\alpha, \beta)$, $-\infty< \alpha < \beta < \infty$,  equipped with the Lebesgue measure.   Assume   $\Phi$ is a $MO$ function on $\Omega$ satisfying the condition {\rm (V)}.
Then the Voltera operator

\[
Au(x) = \int_\alpha^x u(y)\, dy, \ \ \ \ x\in (\alpha,\beta),
\]
 is bounded on $L^\Phi$. 
\end{theorem}

\begin{proof}

By Lemma \ref{lem:111} we get  $\esssup_{x\in\Omega} \Phi^*(x,c) <\infty$ for some $c>0$.
Now by convexity of $\Phi^*$, for all $n\in\N$, $\esssup_{x\in\Omega} \Phi^*\left(x,\frac{c}{n}\right) \le \frac{1}{n}\esssup_{x\in\Omega} \Phi^*(x,c)$. Thus  for for $b$ from the assumptions and some $c_1>0$, 
\[
\esssup_{x\in\Omega} \Phi^*(x,c_1) \le b.
\]
Setting $k(x,y) = \chi_{(\alpha,x)}(y)$, $x,y\in (\alpha, \beta)$, we have $Au(x) = \int_\Omega k(x,y) u(y)\,dy$, $u\in L^0$.  Let $\psi$ be as in Corollary \ref{cor:12}. Then
\begin{align*}
I_\psi(c_1k)& = \int_\alpha^\beta \int_\alpha^\beta \psi((x,y), c_1 \chi_{(\alpha,x)}(y))\,dxdy
=  \int_\alpha^\beta \int_\alpha^\beta \Phi(x,\Phi^*(y,c_1\chi_{(\alpha,x)}(y))\,dxdy\\
&\le  \int_\alpha^\beta \int_\alpha^\beta \Phi(x,b)\,dxdy = (\beta - \alpha) \int_\alpha^\beta \Phi(x,b)\,dx < \infty.
\end{align*}
Hence the kernel $k\in L^\psi$. Finally by Corollary \ref{cor:12}, the Voltera operator is bounded on $L^\Phi$.

\end{proof}

Since an Orlicz function satisfies condition (V), the next result is instant.

\begin{corollary}\label{cor:voltera-orlicz}
Let  $\varphi$ be an Orlicz function and $\Omega=(\alpha,\beta)$, where $-\infty<\alpha<\beta<\infty$. Then the Voltera operator is bounded on Orlicz space $L^\varphi$. 
\end{corollary}

\begin{corollary}\label{cor:voltera-lp}
Let $1\le p(x)<\infty$ a.e. on $(\alpha,\beta)$, where $-\infty<\alpha<\beta<\infty$. Then the Voltera operator is bounded on $L^{p(\cdot)}$. 
\end{corollary}

\begin{proof}
Since  for variable exponent $MO$ function, $\Phi(x,1) = 1/p(x) \le 1$ and $\Phi^*(x,1) = 1/q(x) \le 1$ for a.a. $x\in\Omega$, so in view of Lemma \ref{lem:111},  the conditions in Theorem \ref{th:12} are satisfied and the conclusion holds.

\end{proof}

\begin{corollary}\label{cor:volteralpq}
Let $\Phi$ be a double face functional, that is
\[
\Phi(x,t) = t^{p(x)} + a(x) t^{r(x)},
\]
where $a,p,r$ are real measurable functions on $(\alpha,\beta)$, $-\infty<\alpha<\beta<\infty$, such that $a(x) \ge 0$,  $1 \le p(x)<\infty, \ 1\le r(x) < \infty$ a.e. on $(\alpha, \beta)$. If $a\in L^1$ then the Voltera operator is bounded on $L^\Phi$.
\end{corollary}
\begin{proof}
We have 
\[
\int_\alpha^\beta\Phi(x,1)\,dx = (\beta - \alpha) + \int_\alpha^\beta a(x)\, dx < \infty,
\]
and clearly $\essinf_{x\in\Omega} \Phi(x,1) \ge 1$. In view of Theorem \ref{th:12} we conclude the proof. 

\end{proof}

\section{Copy of $\ell^\infty$}\label{section 4.3}

In this part we give conditions on $\Phi$ in order to $W^{1,\Phi}$  contain a subspace isomorphic to $\ell^\infty$. We start with Musielak-Orlicz space $L^\Phi$.
\begin{proposition} \cite{Kam1984, Kam1998}\label{pr:ellinftyMO}
Let $(\Omega,\Sigma,\mu)$ be a non-atomic measure space. Then a $eMO$ function $\Phi$ does not satisfy condition $\Delta_2$ if and only if  there exists a sequence of bounded and non-negative functions $f_n \in L^\Phi$ such that $f_n \wedge f_m =0$ for $n\ne m$, $I_\Phi(f_n) \le 1/2^n$ and $\|f_n\|_\Phi = 1$ for all $n\in\N$. Consequently,
\[
\left\|\sum_{n=1}^\infty f_n \right\|_\Phi = \|f_n\|_\Phi =1 
\]
for all $n\in\N$.   
\end{proposition}

  The next theorem results directly from Proposition \ref{pr:ellinftyMO}.

\begin{theorem}\label{th:linftyMO}
Let the measure space $(\Omega, \Sigma, \mu)$ be  separable and non-atomic. A Musielak-Orlicz space $L^\Phi$ contains an isomorphic copy of $\ell^\infty$ if and only if $\Phi$ does not satisfy condition $\Delta_2$.
\end{theorem}
\begin{proof}
If $\Phi\notin \Delta_2$, then taking any element $a = \{a_n\}_{n=1}^\infty\in \ell^\infty$ and  the sequence $\{f_n\}_{n=1}^\infty$ from Proposition \ref{pr:ellinftyMO}, we get for every $n\in\N$,
\[
|a_n| = \|a_nf_n\|_\Phi \le\left\|\sum_{n=1}^\infty a_nf_n\right\|_\Phi \le 
\|a\|_\infty.
\]
It follows
\[
\left\|\sum_{n=1}^\infty a_nf_n\right\|_\Phi = 
\|a\|_\infty,
\]
and in fact $\ell^\infty$ is an isometric isomorphic subspace of $L^\Phi$.

  Now assume opposite that $L^\Phi$ has a subspace isomorphic to $\ell^\infty$. Then $L^\Phi$ can not be separable, and so $\Phi\notin\Delta_2$ by Theorem \ref{th:sepMO}.

\end{proof}

\begin{corollary}\label{cor:lpxlinfty}
Let the measure space $(\Omega, \Sigma, \mu)$ be  separable and non-atomic. The variable exponent Lebesgue space $L^{p(\cdot)}$ contains an isomorphic subspace to $\ell^\infty$ if and only if $p^+ = \infty$. 
 \end{corollary}
 \begin{proof}
By Theorem \ref{th:deltalpx}, the variable exponential function satisfies $\Delta_2$ if and only if $p^+<\infty$. Therefore the proof is completed by application of Theorem \ref{th:linftyMO}.

\end{proof}

\begin{theorem} \label{th:ellinftycopyI}
Let $\Phi$ be a $MO$ function on $ (\alpha,\beta)$, $\infty<\alpha<\beta<\infty$ with the Lebesgue measure. If $W^{1,\Phi}$ contains a subspace isomorphic to $\ell^\infty$ then $\Phi$ does not satisfy condition $\Delta_2$.
\end{theorem}
\begin{proof}
 Let $A=\{ (f,f'): f\in W^{1,\Phi}\}$ be a subspace of  the product $L^\Phi \times L^\Phi$. If $L^\Phi\times L^\Phi$ is equipped with the norm $\|(\cdot,\cdot\cdot)\|_{L^\Phi\times L^\Phi}=\|\cdot\|_{\Phi}+\|\cdot\cdot\|_{\Phi}$, then  the space $W^{1,\Phi}$ is isometrically isomorphic to  $A$. Notice that $L^\Phi \times L^\Phi$ is isomorphic to the MO space $L^{\overline{\Phi}}(\Omega\times\{1,2\})$ where $\overline{\Phi}:\Omega \times\{1,2\}\times [0,\infty)\to [0,\infty)$ is defined as
    \[
    \overline{\Phi}(x,y,t)=\Phi(x,t)\chi_{\{1\}}(y)+\Phi(x,t)\chi_{\{2\}}(y),
    \]
for a.a. $x\in\Omega, \ y\in\{1,2\}$ and $t\geq 0$.
      Indeed, the operator $T:L^\Phi \times L^\Phi\to L^{\overline{\Phi}}(\Omega\times \{1,2\})$ defined by 
\[
(T(f_1,f_2))(x,y)=f_1(x)\chi_{\{1\}}(y)+f_2(x)\chi_{\{2\}}(y), \ \ \ \ f_1, f_2 \in L^\Phi,
\]
for a.a. $x\in\Omega$, $y\in\{1,2\}$,
is clearly a linear bijection and its inverse is given by
\[
(T^{-1}f)(x)=(f(x,1),f(x,2)), \ \ \ \ f\in L^{\overline{\Phi}}(\Omega\times \{1,2\}). 
\]
Moreover, taking any $(f_1,f_2)\neq (0,0)$ where $(f_1,f_2)\in L^\Phi\times L^\Phi$,
\begin{align*}
   &I_{\overline{\Phi}}\left(\frac{T(f_1,f_2)}{2\|(f_1,f_2)\|_{L^\Phi\times L^\Phi}}\right)\\&=
   \int\limits_\Omega\Phi\left(x,\frac{|f_1(x)|}{2(\|f_1\|_\Phi+\|f_2\|_\Phi)}\right)dx+\int\limits_\Omega\Phi\left(x,\frac{|f_2(x)|}{2(\|f_1\|_\Phi+\|f_2\|_\Phi)}\right)dx\leq 1.
   \end{align*}
Therefore, for every $(f_1,f_2)\in L^\Phi\times L^\Phi$ we have
\[\|T(f_1,f_2)\|_{L^{\overline{\Phi}}}\leq 2\|(f_1,f_2)\|_{L^\Phi\times L^\Phi}.\]
On the other hand, if $0\neq f\in L^{\overline{\Phi}}(\Omega\times \{1,2\}) $, then for $j=1,2$,
\begin{align*}
  &I_{\Phi}\left(\frac{f(\cdot,j)}{\|f\|_{\overline{\Phi}}}\right)=\int\limits_\Omega \Phi\left(x,\frac{|f(x,j)|}{\|f\|_{\overline{\Phi}}}\right)dx \\&\leq 
\int\limits_\Omega \Phi\left(x,\frac{|f(x,1)|}{\|f\|_{\overline{\Phi}}}\right)dx+ 
\int\limits_\Omega \Phi\left(x,\frac{|f(x,2)|}{\|f\|_{\overline{\Phi}}}\right)dx=
I_{\overline{\Phi}}\left(\frac{f}{\|f\|_{\overline{\Phi}}}\right)\leq 1.
\end{align*}
Hence $\|f(\cdot,j)\|_\Phi \le \|f\|_{\overline\Phi}$  for $j=1,2$.
Consequently,
for every $f\in L^{\overline{\Phi}}(\Omega\times \{1,2\})  $ we have
\[\|T^{-1}f\|_{L^\Phi\times L^\Phi}  = \|f(\cdot,1)\|_\Phi + \|f(\cdot,2)\|_\Phi \leq 2\|f\|_{\overline{\Phi}}.\]
From this we conclude that indeed $L^\Phi \times L^\Phi$ is isomorphic to $L^{\overline{\Phi}}(\Omega\times \{1,2\})$.

  If $W^{1,\Phi}$ contains $\ell^\infty$ isomorphically,  then $L^\Phi\times L^\Phi$ does it too and so $L^{\overline{\Phi}}$ must contain $\ell^\infty$, which implies that ${\overline{\Phi}}$ does not satisfy condition $\Delta_2$ by Theorem \ref{th:linftyMO}. Now we argue by contradiction. If $\Phi$ satisfies $\Delta_2$, then there exist a constant $C>0$ and a non-negative function $h\in L^1(\Omega)$ such that for any $t\geq 0$ and a.a. $x\in\Omega$,
\[\Phi(x,2t)\leq C\Phi(x,t)+h(x).\]
Hence for any $t\geq 0$ and a.a. $(x,y)\in\Omega\times\{1,2\}$ we have
\begin{align*}
  &\overline{\Phi}(x,y,2t)=\Phi(x,2t)\chi_{\{1\}}(y)+\Phi(x,2t)\chi_{\{2\}}(y)\leq\\&
  (C\Phi(x,t)+h(x))\chi_{\{1\}}(y)+(C\Phi(x,t)+h(x))\chi_{\{2\}}(y)=\\&
  C\overline{\Phi}(x,y,t)  + h(x)(\chi_{\{1\}}(y)+\chi_{\{2\}}(y)).
  \end{align*}
The function $H$ given by the formula $H(x,y)=h(x)(\chi_{\{1\}}(y)+\chi_{\{2\}}(y))$ is a non-negative element of $L^1(\Omega\times \{1,2\})$. Therefore ${\overline{\Phi}}$ satisfies the $\Delta_2$ condition, a contradiction.

\end{proof}

\begin{theorem}\label{th:ellinftycopy} Let $\Omega=(\alpha,\beta)$ where $-\infty<\alpha<\beta<\infty$
and $\Phi$ be a $MO$ function on $\Omega$ satisfying condition {\rm (V)}.  If $\Phi$ does not satisfy condition $\Delta_2$ then the Sobolev space $W^{1,\Phi}$ contains a subspace isomorphic to $\ell^\infty$. 
\end{theorem} 

\begin{proof} Let $\{f_k\}\subset L^\Phi$ satisfy the hypothesis of Proposition 
\ref{pr:ellinftyMO}. Since they are bounded on $\Omega$, so $f_k\in L^1$. Define
\begin{equation*}
g_k(x) = \int_\alpha^x f_k(y)\,dy, \ \ \ x\in(\alpha,\beta), \ k\in\N.
\end{equation*}
We have that $g_k\in W^{1,\Phi}$. Indeed $g_k'=f_k \in L^\Phi$, and by the assumption {\rm (V)},   the Voltera operator is bounded on $L^\Phi$,  and so $\|g_k\|_\Phi \le l \|f_k\|_\Phi < \infty$ for some constant $l$. Moreover for every $k\in\N$,
\[
1 = \|f_k\|_\Phi \le \|g_k\|_{1,\Phi} = \|f_k\|_\Phi + \|g_k\|_\Phi \le (1+l) \|f_k\|_\Phi = 1+l.
\]
Analogously, for every $m\in\N$,
\begin{align}\label{eq:12}
1 &\le \left\|\sum_{k=1}^m g_k\right\|_{1,\Phi} \le  \left\|\left(\sum_{k=1}^m g_k\right)'\right\|_\Phi + \left\|\sum_{k=1}^m g_k \right\|_\Phi \\
&= \left\|\sum_{k=1}^m f_k\right\|_\Phi + \left\|\sum_{k=1}^m g_k \right\|_\Phi\le (1+l) \left\|\sum_{k=1}^m f_k\right\|_\Phi \le 1+l.\notag
\end{align}
Hence $\sum_{k=1}^\infty g_k \in W^{1,\Phi}$. Notice that $g_k \ge 0$ and $f_k\ge 0$. Therefore in view of (\ref{eq:12}), for every element $a=(a_k)\in \ell^\infty$, $m\in\N$,
\begin{align}\label{eq:13}
\left\|\sum_{k=1}^m a_k g_k \right\|_{1,\Phi} 
&= \left\|\sum_{k=1}^m a_k g_k' \right\|_\Phi + \left\|\sum_{k=1}^m a_k g_k \right\|_\Phi 
\le \left\|\sum_{k=1}^m |a_k|\,|g_k'| \right\|_\Phi + \left\|\sum_{k=1}^m |a_k|\, |g_k| \right\|_\Phi \\\notag
&= \left\|\sum_{k=1}^m |a_k| f_k\right\|_\Phi + \left\|\sum_{k=1}^m |a_k| g_k \right\|_\Phi
\le \|a\|_\infty \left(\left\|\sum_{k=1}^m f_k\right\|_\Phi + \left\|\sum_{k=1}^m g_k \right\|_\Phi\right)\\\notag
&\le (1+l) \|a\|_\infty\left\|\sum_{k=1}^m f_k\right\|_\Phi \le (1+l) \|a\|_\infty.\notag
\end{align}
On the other hand for every $m, k\in\N$, 
\[
\left\|\sum_{k=1}^m a_k g_k \right\|_{1,\Phi} \ge \left\|\sum_{k=1}^m a_k f_k\right\|_\Phi = \left\|\sum_{k=1}^m |a_k|\, |f_k|\right\|_\Phi \ge \|\,|a_k|\, |f_k|\|_\Phi = |a_k| \, \|f_k\|_\Phi = |a_k|.
\]
Hence for every $m\in\N$,
\begin{equation}\label{eq:14}
\left\|\sum_{k=1}^m a_k g_k \right\|_{1,\Phi} \ge \|a\|_\infty.
\end{equation}
Combining (\ref{eq:13}) and (\ref{eq:14}) we have that $\ell^\infty$ is an isomorphic copy in $W^{1,\Phi}$.

\end{proof}

Since (V) is always satisfied in Orlicz space over $(\alpha, \beta)$, $-\infty < \alpha, \beta < \infty$, we get instantly the following result.

  \begin{corollary}\label{cor:sobpxlinfty}
Let $\Omega=(\alpha,\beta)$ where $-\infty<\alpha<\beta<\infty$. The Orlicz-Sobolev space $W^{1,\varphi}$ contains an isomorphic subspace to $\ell^\infty$ if and only if $\varphi$ does not satisfy condition $\Delta_2^\infty$.
 \end{corollary}

  \begin{corollary}\label{cor:sobpxlinfty}
Let $\Omega=(\alpha,\beta)$ where $-\infty<\alpha<\beta<\infty$. The variable exponent Sobolev space $W^{1,p(\cdot)}$ contains an isomorphic subspace to $\ell^\infty$ if and only if $p^+ = \infty$.
 \end{corollary} 

\begin{proof}
By Corollary \ref{cor:voltera-lp}, the variable exponent function  satisfies condition ${\rm (V)}$. Thus the proof is an immediate consequence of Theorems \ref{th:ellinftycopyI} and \ref{th:ellinftycopy}. 

\end{proof}

The next result follows from Proposition \ref{th:deltadobphase}.

   \begin{corollary}\label{cor:sobdphlinfty}
Let $\Omega=(\alpha,\beta)$ where $-\infty<\alpha<\beta<\infty$. Let the  Sobolev space $W^{1,\Phi}$ be induced by the double phase functional 
\[
\Phi(x,t) = t^{p(x)} + a(x) t^{r(x)}
\]
with $a,p,r$ measurable functions on $\Omega$, such that  $a(x) > 0$  and $p(x )< r(x)$ for a.a. $x\in\Omega$. 

Then $W^{1,\Phi}$ contains an isomorphic subspace to $\ell^\infty$ if and only if $p^+ = \infty$ or $r^+ = \infty$.
 \end{corollary}

 \section{Copy of $\ell^1$}\label{ section 4.4}
 
In this section we will characterize the spaces $W^{1,\Phi}$ that contain an isomorphic copy of $\ell^1$. First we need to recall an analogous result for $MO$ spaces $L^\Phi$.

 \begin{theorem}\label{th:liMO}
 Let $(\Omega,\Sigma,\mu)$ be a non-atomic and separable measure space.  A $MO$ space $L^\Phi$ contains an isomorphic copy of $\ell^1$ if and only if $\Phi$ or $\Phi^*$ do not satisfy condition $\Delta_2$.
\end{theorem}
\begin{proof}
If $\Phi$ does not satisfy condition $\Delta_2$ then $L^\Phi$ contains isomorphically $\ell^\infty$ by Theorem \ref{th:linftyMO},  and so $\ell^1$ is an isomorphic subspace of $L^\Phi$ \cite[Corollary 6.8.]{car}.
If $\Phi^*$ does not satisfy $\Delta_2$ then $L^{\Phi^*}$ contains an isomorphic subspace of $\ell^\infty$ again by Theorem  \ref{th:linftyMO}.  Thus applying Theorem \ref{th:funcMO}, the dual space $(L^\Phi)^* \simeq L^{\Phi^*} \oplus (L^\Phi)_s$ must also contain an isomorphic copy of $\ell^\infty$. Finally by general result in Banach spaces \cite[Proposition 2.e.8]{LT1}, the space $L^\Phi$ must contain a subspace isomorphic to $\ell^1$.

  If $L^\Phi$ contains a subspace isomorphic to $\ell^1$, then it can not be reflexive, and so by Theorem \ref{th:Kothe}, $\Phi$ or $\Phi^*$ does not satisfy $\Delta_2$.

\end{proof}

\begin{corollary}\label{cor:l1lpx}
 Let $(\Omega,\Sigma,\mu)$ be a non-atomic and separable measure space. The variable exponent Lebesgue space $L^{p(\cdot)}$ contains an isomorphic subspace to $\ell^1$ if and only if $p^+ = \infty$ or $p^- = 1$. 
 \end{corollary}
 \begin{proof}
 The proof follows from Theorems \ref{th:deltalpx} and \ref{th:liMO}.

\end{proof}

  \begin{theorem}\label{th:copyl1nec}
Let $\Omega=(\alpha,\beta)$ where $-\infty<\alpha<\beta<\infty$.  If a $MO$ function $\Phi$ and its conjugate $\Phi^*$ satisfy condition $\Delta_2$ then $W^{1,\Phi}$ does not contain an isomorphic subspace of $\ell^1$.
 
 \end{theorem}  

\begin{proof}
If both $\Phi$ and $\Phi^*$ satisfy condition $\Delta_2$ then the space $L^\Phi$ is reflexive by  Theorem \ref{th:Kothe}. Hence $L^\Phi \times L^\Phi$ equipped with norm $\|(\cdot,\cdot\cdot)\|_{L^\Phi\times L^\Phi}=\|\cdot\|_{\Phi}+\|\cdot\cdot\|_{\Phi}$ is reflexive and so $W^{1,\Phi}$ as its closed subspace is reflexive too.  Therefore it can not contain a subspace $\ell^1$.

\end{proof}

\begin{theorem}\label{th:copyl1suf2}
Let $\Omega=(\alpha,\beta)$ where $-\infty<\alpha<\beta<\infty$. Let $\Phi$ be $MO$ function satisfying condition ${\rm (V)}$.  If $W^{1,\Phi}$ does not contain a subspace isomorphic to $\ell^1$ then both $\Phi$ and $\Phi^*$ satisfy $\Delta_2$.
\end{theorem}

\begin{proof}  Let first $\Phi$ do not satisfy $\Delta_2$. By Theorem \ref{th:ellinftycopy}, $\ell^\infty$ is an isomorphic copy in $W^{1,\Phi}$. Thus $\ell^1$ is contained isomorphically in $\ell^\infty$ by \cite[Corollary 6.8.]{car} and so in $W^{1,\Phi}$.

  Assume now that $\Phi^*$ does not satisfy $\Delta_2$. Then in view of Proposition \ref{pr:ellinftyMO}, there exists a sequence of bounded and non-negative  functions $\{f_k\} \subset L^{\Phi^*}$ such that $f_k\wedge f_i = 0$, $k\ne i$, and 
\begin{equation}\label{eq:15}
\left\|\sum_{k=1}^\infty f_k\right\|_{\Phi^*} = \|f_k\|_{\Phi^*}=1, \ \ \ k\in\N.
\end{equation}
Hence
\begin{equation}\label{eq:16}
1 = \|f_k\|_{\Phi^*} \le \|f_k\|_{\Phi^*}^0 \le 2 \|f_k\|_{\Phi^*} = 2, \ \ \ k\in\N.
\end{equation}
By the definition of the Orlicz norm $\|\cdot\|_{\Phi^*}^0 $ for $\epsilon>0$, each $k\in\N$, there exists a non-negative, bounded function $g_k\in L^\Phi$  such that  $\esssupp{g_k} \subset \esssupp{f_k}$, $I_\Phi(g_k) \le 1$ and
\begin{equation}\label{eq:161}
\|f_k\|_{\Phi^*}^0 \le \frac{\epsilon}{2^k} + \int_\Omega f_k(x)\,g_k(x)\,dx.
\end{equation}
Define
\[
h_k(x) = \int_\alpha^x g_k(y)\,dy, \ \ \ x\in(\alpha,\beta), \ k\in\N.
\]
Thus its derivative $h_k' = g_k \in L^\Phi$. Moreover, by the boundedness of the Voltera operator $\|h_k\|_\Phi \le l \|g_k\|_\Phi \le l$, which implies that $h_k$ also belongs to $L^\Phi$. 

  Clearly for every $k\in\N$,
\[
\|h_k\|_{1,\Phi} = \|g_k\|_\Phi + \|h_k\|_\Phi \le 1+l.
\]
Hence  for all $a = (a_k)\in\ell^1$, $m\in\N$,
\[
\left\|\sum_{k=1}^m a_k\,h_k\right\|_{1,\Phi}
\le \sum_{k=1}^m |a_k|\, \|h_k\|_{1,\Phi} \le (1+l) \|a\|_1.
\] 
On the other hand,
\[
\left\|\sum_{k=1}^m a_k\,h_k\right\|_{1,\Phi} =
\left\|\sum_{k=1}^m a_k\,h_k\right\|_\Phi + \left\|\sum_{k=1}^m a_k\,h_k'\right\|_\Phi
\ge \left\|\sum_{k=1}^m a_k\,g_k\right\|_\Phi,
\]
and by H\"older's inequality and by (\ref{eq:15}), (\ref{eq:16}),
\[
\frac12 \int_\Omega \left(\sum_{k=1}^m a_k\, g_k(x)\right) \,\left(\sum_{k=1}^m (sign\,a_k)\,f_k(x)\right)\, dx\le \frac12 \left\|\sum_{k=1}^m a_k\,g_k\right\|_\Phi 2 \left\|\sum_{k=1}^m f_k\right\|_{\Phi^*} \le \left\|\sum_{k=1}^m a_k\,g_k \right\|_\Phi.
\]
Therefore in view of $\esssupp g_k \subset \esssupp f_k$ and of that $f_k$ are disjoint and of (\ref{eq:16}),  (\ref{eq:161}),   we get for every $m\in\mathbb{N}$,
\begin{align*}
\left\|\sum_{k=1}^m a_k\,h_k\right\|_{1,\Phi} &\ge 
\frac12 \int_\Omega \left(\sum_{k=1}^m a_k\, g_k(x)\right) \,\left(\sum_{k=1}^m (sign\,a_k)\,f_k(x)\right)\, dx
= \frac12 \int_\Omega \sum_{k=1}^m |a_k|\, f_k(x)\,g_k(x)\,dx\\ 
&= \frac12 \sum_{k=1}^m |a_k|\, \int_\Omega f_k(x)\,g_k(x)\,dx \ge \frac12\,\sum_{k=1}^m |a_k|\,\left (\|f_k\|^0_{\Phi^*} -\frac{\epsilon}{2^k}\right) \\
&\ge \frac12 \left( \sum_{k=1}^m |a_k| - \sum_{k=1}^m \frac{\epsilon}{2^k}\right)\ge \frac12\left(\sum_{k=1}^m |a_k|-  \epsilon\right).
\end{align*}
Since $\epsilon>0$ and $m\in\mathbb{N}$ were arbitrary, combining the above inequalities we get
\[
\frac12 \|a\|_1 \le \left\|\sum_{k=1}^\infty a_k\,h_k\right\|_{1,\Phi} \le (1+l) \|a\|_1,
\]
which shows that $W^{1,\Phi}$ contains an isomorphic subspace of $\ell^1$ and completes the proof.

\end{proof}

In view of (V)  satisfied in any Orlicz space the next result is an instant corollary of Theorems  \ref{th:copyl1nec} and \ref{th:copyl1suf2}.

\begin{corollary}\label{cor:l1Lorlicz}
Let $\Omega=(\alpha,\beta)$ where $-\infty<\alpha<\beta<\infty$, and $\varphi$ be an Orlicz function.  Then the  space $W^{1,\varphi}$ does not contain isomorphic copy of $\ell^1$ if and only if $\varphi$ and $\varphi^*$ satisfy condition $\Delta_2^\infty$. 
\end{corollary}

\begin{corollary}\label{cor:l1Lpx}
Let $\Omega=(\alpha,\beta)$ where $-\infty<\alpha<\beta<\infty$. The space $W^{1,p(\cdot)}$ does not contain isomorphic copy of $\ell^1$ if and only if $1<p^- \le p^+<\infty$. 
\end{corollary}

\begin{proof}
 We observe first that the Voltera operator is bounded on $L^{p(\cdot)}$ by Corollary \ref{cor:voltera-lp}.  Therefore the conclusion follows by  Theorem \ref{th:deltalpx} and  Theorems \ref{th:copyl1nec} and \ref{th:copyl1suf2}.

\end{proof}


   \begin{corollary}\label{cor:sobdphlinfty}
Let $\Omega=(\alpha,\beta)$ where $-\infty<\alpha<\beta<\infty$. Let the  Sobolev space $W^{1,\Phi}$ be induced by the double phase function 
\[
\Phi(x,t) = t^{p(x)} + a(x) t^{r(x)}
\]
with $a,p,r$ measurable functions on $\Omega$, such that  $a(x) \ge 0$  and $p(x )\le r(x)$ for a.a. $x\in\Omega$. 
If $r^+ < \infty$ and $p^->1$ then 
 $W^{1,\Phi}$ does not contain an isomorphic subspace to $\ell^1$.
\end{corollary} 

\begin{proof}
The Voltera operator is bounded on $L^\Phi$ by Corollary \ref{cor:volteralpq}.  Therefore by Theorems \ref{th:copyl1nec} and \ref{th:copyl1suf2}, $W^{1,\Phi}$  does not contain an isomorphic subspace to $\ell^1$ if and only if  $\Phi$ and $\Phi^*$ satisfy $\Delta_2$.   On the other hand in view of
Proposition \ref{th:deltadobphase} (i), by the assumption that $r^+ < \infty$, $\Phi$ satisfies condition $\Delta_2$, and by the assumption $p^- > 1$, $\Phi^*$ satisfies $\Delta_2$, and the conclusion follows.

\end{proof}

\section{Reflexivity of $W^{1,\Phi}$}\label{section 4.5}

\begin{theorem}\label{th:ref}
Let $\Omega=(\alpha,\beta)$, where $-\infty<\alpha<\beta<\infty$,  and $\Phi$ be a $MO$ function. 

If both $\Phi$ and $\Phi^*$ satisfy condition $\Delta_2$ then $W^{1,\Phi}$ is reflexive.
 
  Let $\Phi$ be $MO$ function  satisfying condition ${\rm (V)}$.   If 
the space $W^{1,\Phi}$ is reflexive then both $\Phi$ and $\Phi^*$ satisfy condition $\Delta_2$. 
\end{theorem}

\begin{proof}
   If both $\Phi$ and $\Phi^*$ satisfy $\Delta_2$ then the $MO$ space $L^\Phi$ is reflexive by Theorem \ref{th:funcMO} (iii), and so is $W^{1,\Phi}$.

  Let assume now condition (V) that  the Voltera operator is bounded on $L^\Phi$. If $W^{1,\Phi}$ is reflexive then it can not contain isomorphic copy of $\ell^\infty$.  Therefore by Theorem \ref{th:ellinftycopy}, $\Phi$ satisfies $\Delta_2$. Similarly $W^{1,\Phi}$  can not contain an isomorphic copy of $\ell^1$, and thus 
 by Theorem \ref{th:copyl1suf2}, $\Phi^*$  also satisfies $\Delta_2$.
 
\end{proof}

Since (V) is fulfilled in any Orlicz space, so we have the following corollary.

\begin{corollary}\label{cor:l1Lorlicz}
Let $\Omega=(\alpha,\beta)$ where $-\infty<\alpha<\beta<\infty$, and $\varphi$ be an Orlicz function.  Then the  space $W^{1,\varphi}$ is reflexive if and only if $\varphi$ and $\varphi^*$ satisfy condition $\Delta_2^\infty$. 
\end{corollary}

\begin{corollary}\label{cor:l1Lpx}
Let $\Omega=(\alpha,\beta)$ where $-\infty<\alpha<\beta<\infty$. The space $W^{1,p(\cdot)}$ is reflexive if and only if $1<p^- \le p^+<\infty$. 
\end{corollary}

\begin{proof}
 We observe first that the Voltera operator is bounded on $L^{p(\cdot)}$ by Corollary \ref{cor:voltera-lp}.  Therefore the conclusion follows by  Theorem \ref{th:ref}.
\end{proof}

\section{Uniform convexity of $W^{1,\Phi}$}\label{section 4.6}

  The concept of uniform convexity of a Banach space $X$ was first introduced by James A. Clarkson in 1936 where the author showed the uniform convexity of $\ell^p$ and $L^p$ spaces for $1<p<\infty$. Since then many authors had studied this property in other instances of Banach spaces  as  Orlicz or Musielak-Orlicz spaces \cite{Chen, Kam1982, H1983}.

  Let $(X,\|\cdot\|)$ be a Banach space equipped with the norm $\|\cdot\|$. We denote by $S(X)$ and $B(X)$ the unit sphere and unit ball of $X$ respectively. We say the $X$ is {\it strictly convex}  whenever $\|\frac{x+y}{2}\| < 1$ for any $x,y \in S(X)$, $x\ne y$. Recall that $X$ is {\it uniformly convex} (for short $UC$) if 
\[
\forall \epsilon \in (0,1)\ \exists \delta > 0\ \forall x,y\in B(X)\ \|x-y\| >\epsilon \Rightarrow \left\|\frac{x+y}{2}\right\| < 1-\delta.
\]

It is not difficult to show that equivalently $X$ is $UC$ whenever for any sequences $\{x_n\}, \{y_n\} \subset B(X) , \|x_n + y_n \| \rightarrow 2$ implies $\|x_n - y_n\| \rightarrow 0$ as $n\to\infty$.

A $MO$ function $\Phi$ is called {\it strictly convex} if for a.a. $x\in\Omega$, the function $t\mapsto \Phi(x,t)$ is strictly convex on $\mathbb{R}_+$, that is 
\begin{align*}
\exists A\subset\Omega,\ \mu(A) = 0\  \forall x\in \Omega\setminus A ,\forall u\ne v \in \R_+ \ \forall \lambda\in (0,1)\\ \Phi(x, \lambda u + (1-\lambda)v) < \lambda\Phi(x,u) + (1-\lambda) \Phi(x,v).
\end{align*}
Following \cite{Chen, H1983} we say that a $MO$ function $\Phi$ is {\it uniformly convex} if
\begin{align}\label{def:UC}
&\forall \epsilon > 0 \ \exists \delta>0 \ \exists \ 0\le f\in L^0 ,\ \int_\Omega \Phi(x, f(x))\,dx \le \epsilon,\\ \notag
 &\text{and whenever} \ \ \forall u,v \ge 0\ \forall a.a.\ x\in\Omega \ |u-v| \ge \epsilon \max\{u,v\} \ \ \text{and} \ \ \ |u-v| \ge f(x), \\ \notag
&\text{then}\ \ \Phi\left(x, \frac{u+v}{2}\right) \le \frac{1-\delta}{2} (\Phi(x,u) + \Phi(x,v)).
\end{align}

  It is not difficult to show that $\Phi$ is uniformly convex if and only if for every $\epsilon > 0$,
\begin{equation}\label{eq:unconv}
\lim_{c\to 0} \int_\Omega \Phi(x,P_{\epsilon, c}(x))\, dx = 0,
\end{equation}
where 
\begin{align}\label{eq:unconv2}
P_{\epsilon, c}(x) &= \sup\{u-v: (u,v) \in E_{\epsilon, c, x}\}, \ \ \text{with}\\\notag
E_{\epsilon, c, x} &= \{(u,v): u,v \ge 0, |u-v| > \epsilon \max\{u,v\}, \ \text{and}\\\notag
& \Phi\left(x,\frac{u+v}{2}\right) > \frac12  (1-c) (\Phi(x,u) + \Phi(x,v))\}.\notag
\end{align}

  First version of uniform convexity for an Orlicz function $\varphi$ was defined in the doctoral thesis of W.A.J. Luxemburg \cite{Lux}. He showed that $L^\varphi$ is uniformly convex if $\varphi$  satisfies condition $\Delta_2$ and is uniformly convex. Later in \cite{Kam1982}, three versions of uniform convexity of $\varphi$ have been introduced,  one for the case of non-atomic infinite measure, another one for non-atomic finite measure and  one for counting measure. It was also proved that appropriate $\Delta_2$ condition and uniform convexity of $\varphi$ are necessary and sufficient conditions of uniform convexity of $L^\varphi$.
Combining all three conditions  in the case of $MO$ function results in the present form (\ref{def:UC}) where  we have a function $f$ with small modular $I_\Phi(f) \le \epsilon$. If $f= 0$ a.e. on $\Omega$ then in fact the inequality defining uniform convexity of $\Phi$ is the same condition as in the original paper \cite{Lux}, uniform for every parameter $x\in \Omega$.

  The definition of uniform convexity  of $\Phi$ seems to be quite complicated, but in the next remark  is explained that this condition  not only implies  strict convexity of $\Phi$, but something more, a sort  of uniform strict convexity.  In fact it can be expressed by the uniform inequalities of  ratios of its derivatives. This was first observed in \cite{Ak} for  Orlicz functions,  and later in \cite{HK} for $MO$ functions.

\begin{remark}\label{rem:UC}
{\rm (1)}  It is standard to show that a uniformly convex function $\Phi$ is  strictly convex on $\R_+$ for a.e. $x\in\Omega$.

  {\rm (2)} \cite{HK} The condition for $\Phi$ being uniformly convex can be expressed in terms of its derivatives. Let $\Phi'(x,t)$ denote  the right derivative of $\Phi(x,t)$ with respect to $t>0$ for a.a. $x\in\Omega$. Then $\Phi$ is uniformly convex if  
for every $\epsilon > 0$ there exists a constant $k_\epsilon > 1$ such that
\begin{equation}\label{eq:UC}
 \Phi'(x,(1+\epsilon) t) \ge k_\epsilon \Phi'(x, t)
 \end{equation}
for a.a. $x\in\Omega$, every $t\ge 0$.
\end{remark}

  Characterization of $UC$ of $L^\Phi$ was considered in 
\cite {H1983} and later in \cite{Chen}.  Partial results on uniform convexity of $L^\Phi$ or $L^{p(\cdot)}$ have been given in Section 2.4 in \cite{DHHR}. Recall the complete characterization of $UC$ in $L^\Phi$.

\begin{theorem}\label{th:uniconmosp}\cite[Theorem 5.15]{Chen}
Let $(\Omega, \Sigma, \mu)$ be a non-atomic separable measure space. The MO space $(L^\Phi, \|\cdot\|_\Phi)$ is uniformly convex if and only if  $\Phi$ satisfies $\Delta_2$ and   $\Phi$ is uniformly convex.
\end{theorem}

\begin{lemma}\label{lem:unifconLpx}
Let $\Phi(x,t)= \frac{t^{p(x)}}{p(x)}$, $t\ge 0$, $1\le p(x) < \infty$ a.a.  $x\in\Omega$. If $p^{-}>1$  then $\Phi$ is uniformly convex.
\end{lemma}
\begin{proof}
 Let $p^- >1$. Then for a.a. $x\in\Omega$,
\[
 \Phi'(x,(1+\epsilon) t)  = (1  + \epsilon)^{p(x)} t^{p(x)-1} \ge (1 + \epsilon)^{p^-} t^{p(x)-1} = k_\epsilon \Phi'(x, t),
 \]
where $k_\epsilon = (1 + \epsilon)^{p^-} > 1$. Thus we showed (\ref{eq:UC}), and $\Phi$ is uniformly convex.

\end{proof}

\begin{lemma}\label{unifconv}
Let $\Phi: \Omega\times \mathbb{R}_+ \to \mathbb{R}_+$ be a double phase functional, that is 
\begin{equation}\label{eq:dblfun}
\Phi(x,t) = t^{p(x)} + a(x) t^{r(x)},
\end{equation}
where $a,p,r$ are measurable functions on $\Omega$,  $a(x) \ge 0$ a.e. and $1\le p(x) \le r(x) <\infty$ a.e.. If $p^->1$ then $\Phi$ is uniformly convex.  
\end{lemma}

\begin{proof}
Let $p^-> 1$. Then for a.a. $x\in\Omega$, $t\ge 0$,
\begin{align*}
 \Phi'(x,(1+\epsilon) t)  &= p(x) (1  + \epsilon)^{p(x)-1} t^{p(x)-1}  + a(x) r(x) (1+\epsilon)^{r(x) -1} t^{r(x) - 1}\\
 &\ge (1 + \epsilon)^{p^- - 1} p(x) t^{p(x) -1}  + a(x) r(x) (1+\epsilon)^{p^- -1} t^{r(x) - 1} =  (1 + \epsilon)^{p^- - 1}\Phi'(x, t),
 \end{align*}
 where  $(1 + \epsilon)^{p^- - 1}>1$. It follows uniform convexity of $\Phi$ by (\ref{eq:UC})

\end{proof}

Partial results of the next corollary are known \cite{DHHR}, but here we present a complete criterion of uniform convexity of $L^{p(\cdot)}$.
\begin{corollary} \label{cor:uclpx} Let $(\Omega, \Sigma, \mu)$ be a non-atomic separable measure space.
The variable exponent Lebesgue space $L^{p(\cdot)}$ is uniformly convex if and only if $1<p^-\le p^+<\infty$.
\end{corollary}

\begin{proof}
If $1<p^- \le p^+ <\infty$ then $\Phi(x,t) = \frac{t^{p(x)}}{p(x)}$ satisfies $\Delta_2$ by Theorem \ref{th:deltalpx} and is uniformly convex by Lemma \ref{lem:unifconLpx}. Thus the space is uniformly convex in view of Theorem \ref{th:uniconmosp}. On the other hand uniform convexity of $L^{p(\cdot)}$ implies reflexivity of the space, and so it cannot have an isomorphic subspace of  $\ell^1$, and thus $1<p^- \le p^+ < \infty$ by Corollary \ref{cor:l1lpx}.

\end{proof}

\begin{corollary} \label{cor:uclpr} Let $(\Omega, \Sigma, \mu)$ be a non-atomic separable measure space. If $\Phi$ is a double phase functional of the form {\rm(\ref{eq:dblfun})}, then the space $L^\Phi$ is uniformly convex
 if $1<p^-$ and $r^+<\infty$.
\end{corollary}
\begin{proof} If $p^- > 1$ then by Lemma \ref{unifconv}, $\Phi$ is uniformly convex. If $r^+ < \infty$ then $\Phi$ satisfies $\Delta_2$ by Theorem \ref{th:deltadobphase}. We conclude in view of Theorem \ref{th:uniconmosp}.

\end{proof}

Now we proceed to consider uniform convexity of $MOS$ space $W^{1,\Phi}$.

\begin{theorem}\label{th:UC}
Let $\Omega=(\alpha,\beta)$ where $-\infty<\alpha<\beta<\infty$. Let $\Phi$ be a $MO$ function and $W^{1,\Phi}$ be the Sobolev space equipped with the norm  $\|f\|_{1,\Phi} = \|f\|_\Phi + \|f'\|_\Phi$, $f\in W^{1,\Phi}$.

   If $\Phi$ satisfies condition $\Delta_2$ and  $\Phi$ is uniformly convex then 
the space $W^{1,\Phi}$   is uniformly convex.

  If in addition $\Phi$ satisfies ${\rm (V)}$,
then $\Delta_2$ and uniform convexity of $\Phi$ are also necessary conditions for the space $W^{1,\Phi}$ to be uniformly convex.

\end{theorem}

\begin{proof}
By Theorem \ref{th:uniconmosp} if $\Phi$ satisfies condition $\Delta_2$ and $\Phi$ is uniformly convex then the space $L^\Phi$ is uniformly convex. Therefore $L^\Phi\times L^\Phi$  equipped with norm (\ref{eq:norm}) is also uniformly convex, and so is $W^{1,\Phi}$.

  Let now $\Phi$ satisfy (V) and $W^{1,\Phi}$ be uniformly convex. Then the space $W^{1,\Phi}$ is reflexive, and so it can not have a subspace isomorphic to $\ell^\infty$. 
By Theorem \ref{th:ellinftycopy}, $\Phi$ needs to satisfy condition $\Delta_2$.

  Thus assume that $\Phi$ satisfies $\Delta_2$ and $\Phi$ is not uniformly convex. It follows by (\ref{eq:unconv}) that there exist $\epsilon>0$ and a sequence $\{c_k\} \subset (0,1)$ with $\lim_{k\to\infty} c_k = 0$ and
\[
\lim_{k\to\infty} \int_\Omega \Phi(x, P_{\epsilon, c_k}(x))\, dx > 0.
\]
Hence there are $\delta >0$, $N>0$ such that for all $k>N$,
\begin{equation}\label{eq:17}
\int_\Omega \Phi(x,P_{\epsilon,c_k}(x))\, dx > 2\delta.
\end{equation}
In view of (\ref{eq:unconv2}) for every $c\in(0,1)$ there exists a sequence
 $\{u_k^c,v_k^c\}$ of non-negative measurable functions satisfying the following conditions,
\begin{equation}\label{eq:18}
 |u_k^c(x) -v_k^c(x)| \uparrow P_{\epsilon,c} (x) \ \ \text{if} \ \ k\to\infty, \ \ \text{for}\ \  a.a. \ \  x\in\Omega.
\end{equation}
\begin{align*}
&\text{If} \ \  u_k^c(x) \ne v_k^c(x)\ \ \text{then}\ \  
\Phi(x, |u_k^c(x) - v_k^c(x)|) \ge \max\{\Phi(x, \epsilon u_k^c(x)), \Phi(x, \epsilon v_k^c(x))\} \ \ \ \text{and}\\ 
&\Phi\left(x, \frac{u_k^c(x) + v_k^c(x)}{2}\right) > \frac{1-c}{2} (\Phi(x,u_k^c(x)) + \Phi(x,v_k^c(x))).
\end{align*}
By (\ref{eq:17}) and (\ref{eq:18}), for every $k\in\N$ there exsts $j_k\in\N$ such that for all $k\in\N$,
\[
\int_\Omega\Phi(x,|u_{j_k}^{c_k}(x) - v_{j_k}^{c_k}(x)|)\, dx > \frac12 \int_\Omega \Phi(x, P_{\epsilon,c_k} (x))\,dx > \delta.
\]
Letting $u_k = u_{j_k}^{c_k}$, $v_k = v_{j_k}^{c_k}$ and applying the above inequalities we obtain for every $k\in\N$,
\begin{equation}\label{eq:19}
 I_\Phi(u_k-v_k) =\int_\Omega \Phi(x, |u_k(x) - v_k(x) | )\, dx > \delta,
\end{equation}
and if $u_k(x)\ne v_k(x)$ then
\begin{equation}\label{eq:20}
\Phi(x, |u_k(x) - v_k(x)|) > \max\{ \Phi(x, \epsilon u_k(x),\Phi(x, \epsilon v_k(x)\},
\end{equation}
\begin{equation}\label{eq:21}
\Phi\left(x, \frac12 (u_k(x) + v_k(x))\right) > \frac{1-c_k}{2} (\Phi(x,u_k(x) + \Phi(x, v_k(x))).
\end{equation}
In view of $\Delta_2$, by Theorem \ref{th:ordconMO}, there exists $\gamma\in (0,\delta)$ such that for all $u\in L^\Phi$,
\begin{equation}\label{eq:22}
I_\Phi(u) < \gamma\Rightarrow \|u\|_\Phi < \epsilon.
\end{equation}
 By (\ref{eq:19}) for every $k\in\N$ we find the sets $E_k$  satisfying
 \[
 E_k \subset \{x\in\Omega: u_k(x)\ne v_k(x)\} \ \ \ \text{and}\ \ \ \int_{E_k} \Phi(x, |u_k(x) - v_k(x)|)\,dx = \gamma.
 \]
If $x\in E_k$ then $u_k(x) \ne v_k(x)$ and  by (\ref{eq:20}),  $\Phi(x, |u_k(x) - v_k(x) ) > \Phi(x, \epsilon u_k(x))$. Hence 
 \begin{equation}\label{eq:23}
 \int_{E_k} \Phi(x, \epsilon u_k(x))\,dx \le \int_{E_k} \Phi(x, |u_k(x) - v_k(x)|)\, dx = \gamma.
 \end{equation}
 It follows in view of (\ref{eq:22}) that $\|\epsilon u_k \chi_{E_k}\|_\Phi <\epsilon$. Consequently and by symmetry, for all $k\in\N$,
 \[
 \|u_k \chi_{E_k}\|_\Phi \le 1 \ \ \ \text{and} \ \ \ \|v_k \chi_{E_k} \|_\Phi \le 1.
 \]
 Let for each $k\in\N$,
 \[
 \Omega_k = \{x\in E_k: \Phi(x,u_k(x)) \ge \Phi(x, v_k(x))\} \ \ \ \text{and} \ \ \ \widetilde{\Omega}_k = E_k \setminus \Omega_k.
 \] 
Set also
\[
\alpha_k = \int_{\Omega_k} (\Phi(x,u_k(x)) - \Phi(x,v_k(x)))\, dx \ \ \ \text{and}\ \ \ \widetilde{\alpha}_k = \int_{\widetilde{\Omega}_k} (\Phi(x,v_k(x)) - \Phi(x,u_k(x)))\, dx.
\]
Clearly, $\alpha_k, \,  \widetilde{\alpha}_k >0$. Therefore for each $k\in\N$ there exists $F_k \subset \Omega_k$ satisfying
\[
\int_{F_k} (\Phi(x,u_k(x)) - \Phi(x,v_k(x)))\, dx = \frac{\alpha_k}{2}.
\] 
Hence
\[
\int_{F_k} (\Phi(x,u_k(x)) - \Phi(x,v_k(x)))\, dx = 
\int_{\Omega_k \setminus F_k} (\Phi(x,u_k(x)) - \Phi(x,v_k(x)))\, dx,
\]
which implies that 
\[
\int_{F_k} \Phi(x,v_k(x))\,dx + \int_{\Omega_k \setminus F_k} \Phi(x,u_k(x))\, dx =  \int_{\Omega_k \setminus F_k} \Phi(x,v_k(x))\, dx + \int_{F_k} \Phi(x,u_k(x))\, dx.
\]
Analogously for every $k\in\N$ there exists $\widetilde{F}_k \subset \widetilde{\Omega}_k$ such that
\[
\int_{\widetilde{F}_k} \Phi(x,v_k(x))\,dx + \int_{\widetilde{\Omega}_k \setminus \widetilde{F}_k} \Phi(x,u_k(x))\, dx =  \int_{\widetilde{\Omega}_k \setminus \widetilde{F}_k} \Phi(x,v_k(x))\, dx + \int_{\widetilde{F}_k} \Phi(x,u_k(x))\, dx.
\]
Setting now for $k\in\mathbb{N}$,
\begin{align*}
\hat{x}_k &= u_k \chi_{F_k \cup (\widetilde{\Omega}_k \setminus \widetilde{F}_k)} +v_k \chi_{\widetilde{F}_k \cup (\Omega_k \setminus F_k)},\\
\hat{y}_k &= v_k \chi_{F_k \cup (\widetilde{\Omega}_k \setminus \widetilde{F}_k)} +u_k \chi_{\widetilde{F}_k \cup (\Omega_k \setminus F_k)},
\end{align*}
we get
\[
[F_k \cup (\widetilde{\Omega}_k \setminus \widetilde{F}_k)] \cup [\widetilde{F}_k \cup (\Omega_k \setminus F_k)] = E_k, \  \ \
[F_k \cup (\widetilde{\Omega}_k \setminus \widetilde{F}_k)] \cap [\widetilde{F}_k \cup (\Omega_k \setminus F_k)] = \emptyset \ \ \ 
 \text{and}  \ \ \  I_\Phi(\hat{x}_k) = I_\Phi(\hat{y}_k)
\]
 for all $k\in\N$. In view of (\ref{eq:20}) and (\ref{eq:23}),
\[
 \delta > \gamma = I_\Phi((u_k - v_k)\chi_{E_k}) \ge 
I_\Phi(\epsilon \max\{u_k,v_k\} \chi_{E_k}),
\]
and so by (\ref{eq:22}), $\|\epsilon \max\{u_k,v_k\} \chi_{E_k}\|_\Phi \le \epsilon$. Hence $\|\max\{u_k,v_k\} \chi_{E_k}\|_\Phi \le 1$ and consequently
\[
I_\Phi(\max\{u_k, v_k\}\chi_{E_k}) \le 1.
\]
Since $\hat{x}_k \le \max\{u_k,v_k\}\chi_{E_k}$ and 
$\hat{y}_k \le \max\{u_k,v_k\}\chi_{E_k}$, we get
\[
\beta_k := I_\Phi(\hat{x}_k) = I_\Phi(\hat{y}_k) \le I_\Phi(\max\{u_k, v_k\}\chi_{E_k}) \le 1.
\]
Now for every $k\in\mathbb{N}$, in view of $\mu(\Omega\setminus E_k) > 0$, there exist $G_k\subset \Omega\setminus E_k$ and $\sigma_k >0$ such that
\[
\int_{G_k} \Phi(x, \sigma_k)\,dx = 1- \beta_k.
\]
Finally let
\[
x_k = \hat{x}_k \chi_{E_k} + \sigma_k \chi_{E_k}, \ \ \ 
y_k = \hat{y}_k \chi_{E_k} + \sigma_k \chi_{E_k}.
\]
Then for all $k\in\N$,
\begin{equation}\label{eq:25}
I_\Phi(x_k) = I_\Phi(y_k) = 1.
\end{equation} 
Moreover by (\ref{eq:23}),
\begin{align*}
0<\gamma &= I_\Phi((u_k - v_k)\chi_{E_k}) = I_\Phi(x_k - y_k) = I_\Phi(\hat{x}_k - \hat{y}_k) \\
&\le I_\Phi(\max\{\hat{x}_k, \hat{y}_k\} )= I_\Phi(\max\{u_k, v_k\} \chi_{E_k}) \le 1.
\end{align*}
Consequently for all $k\in\N$,
\begin{equation}\label{eq:26}
0< \gamma \le I_\Phi(x_k - y_k) \le \|x_k - y_k\|_\Phi.
\end{equation}
Since 
$\|x_k\|_\Phi = \|y_k\|_\Phi = 1$, $\|\frac{x_k + y_k}{2}\|_\Phi \le 1$. Hence 
\begin{equation}\label{eq:27}
\left\|\frac{x_k + y_k}{2}\right\|_\Phi \ge I_\Phi\left(\frac{x_k + y_k}{2}\right).
\end{equation} 
Moreover,
\begin{align*}
I_\Phi\left(\frac{x_k + y_k}{2}\right) &= I_\Phi\left(\frac{\hat{x}_k + \hat{y}_k}{2}\chi_{E_k}\right) + I_\Phi(\sigma_k \chi_{E_k})\\
&=\int_\Omega \Phi\left(x, \frac{u_k(x) + v_k(x)}{2} \chi_{F_k \cup (\widetilde{\Omega}_k \setminus \widetilde{F}_k)} (x)\right)\, dx \\
&+ \int_\Omega \Phi\left(x, \frac{u_k(x) + v_k(x)}{2} \chi_{\widetilde{F}_k \cup (\Omega_k \setminus F_k)} (x)\right)\, dx + \int_{G_k} \Phi(x, \sigma_k)\, dx\\
&=\int_{E_k} \Phi\left(x, \frac{u_k(x) + v_k(x)}{2} \right)\,dx + \int_{G_k} \Phi(x, \sigma_k)\, dx\\
&\ge \frac{1-c_k}{2} \int_{E_k} (\Phi(x,u_k(x)) + \Phi(x,v_k(x)))\,dx + \int_{G_k} \Phi(x, \sigma_k)\, dx\ \ \ \ \text{by} \ \ (\ref{eq:21})\\
&=\frac12 I_\Phi(x_k) + \frac12 I_\Phi(y_k) - \frac{c_k}{2} \int_{E_k}  (\Phi(x,u_k(x)) + \Phi(x,v_k(x)))\,dx \\
&\ge 1-c_k \to 1 \ \ \  \text{by} \ \ \ (\ref{eq:25}),
\end{align*}
when $k\to\infty$.

  Combining the above and (\ref{eq:26}), (\ref{eq:27}), it follows that $L^\Phi$ is not uniformly convex. 

  Now we will proceed to show that  $W^{1,\Phi}$ is not uniformly convex either. 

  Recall that a measurable function is called {\it simple} if it assumes finite number of values. A function $f:
(\alpha, \beta) \to \C$ is called a {\it step function} if there exists a finite partition $\alpha=\alpha_0 < \alpha_1 < \dots < \alpha_m = \beta$ and the numbers $\{a_i\}_{i=1}^m\subset \C$ such that $f(x) = \sum_{i=1}^m a_i \chi_{(\alpha_{i-1}, \alpha_i)} (x)$, $x\in (\alpha, \beta)$.

  First observe that the functions $u_k^c$, $v_k^c$ satisfying (\ref{eq:18}) can be chosen to be simple functions. Therefore the functions $x_k$ and $y_k$ can be also chosen as simple functions. The next observation is that these functions can be replaced by step functions. In fact it follows from the regularity of the Lebesgue measure on $\Omega = (\alpha, \beta)$ and the assumption of $\Delta_2$ condition. By Theorem \ref{th:ordconMO}, $L^\Phi$ is order continuous under the assumption of $\Delta_2$. It implies in particular that for any $f\in L^\Phi$ and every $\epsilon>0$ there is $\delta>0$ such that whenever $\mu(A) < \delta$ then $\|f\chi_A\|_\Phi < \epsilon$ \cite{BS}. By regularity of the Lebesgue measure, for any measurable $A\subset \Omega$ with $\chi_A \in L^\Phi$ and any $\delta>0$, there exist disjoint open intervals $G_1, \dots, G_m$ such that $\mu((A \setminus \cup_{i=1}^m G_i)\cup( \cup_{i=1}^m G_i \setminus A)) < \delta$. Hence $\|\chi_A - \chi_{\cup_{i=1}^m G_i }\|_\Phi = \|\chi_{(A \setminus \cup_{i=1}^m G_i)\cup( \cup_{i=1}^m G_i \setminus A)}\|_\Phi < \epsilon$. Therefore we can approximate any measurable subset of $\Omega$ by a finite union of open disjoint intervals. So any simple function can be replaced by a step function. It follows that the functions $x_k$ and $y_k$ can be taken as step functions. Recall that $x_k$ and $y_k$ are non-negative.

  By the above discussion, without loss of generality, assume for every $k\in\N$,
\[
x_k = a_{k1} \chi_{(\alpha_0, \alpha_1)} + a_{k2} \chi_{(\alpha_1, \alpha_2)} + \dots + a_{kM_k} \chi_{(\alpha_{M_k-1}, \alpha_{M_k})},
\]
where $\alpha = \alpha_0 < \alpha_1 < \dots < \alpha_{M_k} = \beta$ and $\{a_{ki}\}_{i=1}^{M_k}\subset [0,\infty)$. Similarly let
\[
y_k = b_{k1} \chi_{(\beta_0, \beta_1)} + b_{k2} \chi_{(\beta_1, \beta_2)} + \dots + b_{kJ_k} \chi_{(\beta_{J_k-1}, \beta_{J_k})},
\]
with $\alpha = \beta_0 < \beta_1 < \dots < \beta_{J_k} = \beta$ and $\{b_{kj}\}_{j=1}^{J_k}\subset [0,\infty)$. 
Now let  $\{(\gamma_{i-1}, \gamma_i)\}_{i=1}^m$ be the family of all intersections of the intervals $(\alpha_{p-1}, \alpha_p) \cap  (\beta_{j-1}, \beta_j) $, $p=1,\dots,M_k$, $j=1,\dots,J_k$,  which are neither empty nor reduced to  one point. Let $\gamma_i$  be  ordered as $\alpha = \gamma_0 < \gamma_1 < \dots <\gamma_m = \beta$. The numbers $\gamma_i$ and $m\in\N$ depend on $k$. Both functions $x_k$ and $y_k$ are constant on every interval $(\gamma_{i-1}, \gamma_i)$.

  Let $l\in \{1,\dots,m\}$ be fixed.  Divide $(\gamma_{l-1},\gamma_l)$ into $2n_l$ equal subintervals for $n_l\in\mathbb{N}$, 
\[
(\gamma^l_1,\gamma^l_2), (\gamma^l_2,\gamma^l_3), \dots , (\gamma^l_{2n_l},\gamma^l_{2n_l+1}),
\]
such that
\begin{align*}
\int_{\gamma_{2i-1}^l}^{\gamma_{2i}^l} \max\{x_k, y_k\} &< \frac{1}{2^k}, \ \ \ \ i= 1,\dots,n_l,\\
\int_{\gamma_{2i}^l}^{\gamma_{2i+1}^l} \max\{x_k, y_k\} &< \frac{1}{2^k}, \ \ \ \ i= 1,\dots,n_l.
\end{align*}
Since each $x_k$ or $y_k$ is constant on the interval $(\gamma_{l-1},\gamma_l)$, so 

\begin{equation*}
\int_{\gamma_{2i-1}^l}^{\gamma_{2i}^l} x_k = \int_{\gamma_{2i}^l}^{\gamma_{2i+1}^l} x_k ,\ \ \ \ \int_{\gamma_{2i-1}^l}^{\gamma_{2i}^l} y_k = \int_{\gamma_{2i}^l}^{\gamma_{2i+1}^l} y_k, \ \ \   
i= 1,\dots,n_l.\\
\end{equation*}

  Define on each interval $(\gamma_{l-1},\gamma_l)$, $l=1,\dots, m$,
\[
\tilde{x}_k = x_k \chi_{(\gamma_1^l, \gamma_2^l)} -x_k \chi_{(\gamma_2^l, \gamma_3^l) }
+ \dots + x_k \chi_{(\gamma_{2n_l-1}^l, \gamma_{2n_l}^l)}- x_k \chi_{(\gamma_{2n_l}^l, \gamma_{2n_l+1}^l)},
\]
\[
\tilde{y}_k = y_k \chi_{(\gamma_1^l, \gamma_2^l)} -y_k \chi_{(\gamma_2^l, \gamma_3^l) }
+ \dots + y_k \chi_{(\gamma_{2n_l -1}^l, \gamma_{2n_l}^l)} - y_k \chi_{(\gamma_{2n_l}^l, \gamma_{2n_l+1}^l)}.
\]
We defined $\tilde{x}_k$, $\tilde{y}_k$ on every $(\gamma_{l-1},\gamma_l)$, so they are well defined on $(\alpha,\beta)$.
 If $x\in (\gamma_{l-1}, \gamma_l)$ then either $x\in (\gamma_{2i-1}^l, \gamma_{2i}^l)$  or  $x\in (\gamma_{2i}^l, \gamma_{2i+1}^l)$ for some $i=1,\dots, n_l$. For $x\in (\gamma_{2i-1}^l, \gamma_{2i}^l)$,
\[
\left|\int_{\gamma_{l-1}}^x \tilde{x}_k \right| = \left| \left(\int_{\gamma_1^l}^{\gamma_2^l} x_k - \int_{\gamma_2^l}^{\gamma_3^l} x_k\right) + \dots + \int_{\gamma_{2i-1}^l}^x x_k \right| \le \left|\int_{\gamma_{2i-1}^l}^{\gamma_{2i}^l} x_k \right| < \frac{1}{2^k}.
\]
For $x\in (\gamma_{2i}^l, \gamma_{2i+1}^l)$,
\begin{align*}
\left|\int_{\gamma_{l-1}}^x \tilde{x}_k \right| &= \left| \left(\int_{\gamma_1^l}^{\gamma_2^l} x_k - \int_{\gamma_2^l}^{\gamma_3^l} x_k\right) + \dots + \left(\int_{\gamma_{2i-1}^l}^{\gamma_{2i}^l} x_k - \int_{\gamma_{2i}^l}^x x_k\right) \right| \\
&= \int_{\gamma_{2i-1}^l}^{\gamma_{2i}^l} x_k - \int_{\gamma_{2i}^l}^x x_k < \frac{1}{2^k}.
\end{align*}
Combining the above we get for every $k\in\N$ and $x\in(\gamma_{l-1}, \gamma_l)$,
\[
\int_{\gamma_{l-1}}^{\gamma_l} \tilde{x}_k = 0, \ \ \ \ \left|\int_{\gamma_{l-1}}^x \tilde{x}_k\right| < \frac{1}{2^k}.
\] 
Similarly we get for every $k\in\N$ and $x\in(\gamma_{l-1}, \gamma_l)$,
\[
\int_{\gamma_{l-1}}^{\gamma_l} \tilde{y}_k = 0, \ \ \ \ \left|\int_{\gamma_{l-1}}^x \tilde{y}_k\right| < \frac{1}{2^k}.
\] 
Since the above inequalities are satisfied for every $l=1,\dots,m$ we obtain for every $x\in (\alpha, \beta)$,
\begin{equation}\label{eq:28}
\left| \int_\alpha^x \tilde{x}_k \right| < \frac{1}{2^k}, \ \ \ \ \ \left|\int_\alpha^x \tilde{y}_k \right| < \frac{1}{2^k}.
\end{equation}
Let for $k\in\N$, $x\in (\alpha, \beta)$,
\[
f_k(x) =\int_\alpha^x \tilde{x}_k, \ \ \ \ g_k(x) = \int_\alpha^x \tilde{y}_k.
\]
Then
\[
f_k'(x) = \tilde{x}_k(x), \ \ \ \ g_k'(x) = \tilde{y}_k(x),
\]
for a.a. $x\in (\alpha,\beta)$.
By (\ref{eq:28}) for every $\lambda>0$,
\begin{equation}\label{eq:29}
I_\Phi(\lambda f_k) = \int_\Omega \Phi\left(x, \lambda\left|\int_\alpha^x \tilde{x}_k\right|\right)\, dx \le \frac{1}{2^k} \int_\Omega \Phi(x,\lambda)\,dx.
\end{equation}
Now by  assumption ${\rm (V)}$ (see Definition \ref{V condtion}) and $\Delta_2$,  $\int_\Omega \Phi(x, \lambda)\, dx <\infty$. Hence  the right side of (\ref{eq:29}) tends to zero when $k\to \infty$. It follows that 
\[
\|f_k\|_\Phi \to 0\ \ \ \ \text{if}\ \ \ \ k\to \infty.
\]
 Similarly 
\[
 \|g_k\|_\Phi\to 0 \ \ \ \ \text{as} \ \ \ \ k\to\infty.
\]
In view of (\ref{eq:25}), we have for every $k\in\N$, 
\[
\|\tilde{x}_k\|_\Phi=  \|\,|\tilde{x}_k|\, \|_\Phi = \|x_k\|_\Phi =1 \ \ \ \text{and} \ \ \  \|\tilde{y}_k\|_\Phi=  \|\,|\tilde{y}_k|\, \|_\Phi = \|y_k\|_\Phi =1.
\]
Consequently,
\[
\|f_k\|_{1,\Phi} = \|f_k\|_\Phi + \|f_k'\|_\Phi = \|f_k\|_\Phi + \|\tilde{x}_k\|_\Phi \to 1, 
\]
\[ 
\|g_k\|_{1,\Phi} = \|g_k\|_\Phi + \|g_k'\|_\Phi = \|g_k\|_\Phi + \|\tilde{y}_k\|_\Phi \to 1,
\]
as $k\to\infty$.

  Moreover $\left|\frac{\tilde{x}_k + \tilde{y}_k}{2}\right| = \frac{x_k + y_k}{2}$ for every $k\in\N$. Thus in view of (\ref{eq:27}),
\begin{align*}
\left\|\frac{f_k + g_k}{2}\right\|_{1,\Phi} &= \left\|\frac{f_k + g_k}{2}\right\|_\Phi + \left\|\frac{\tilde{x}_k + \tilde{y}_k}{2}\right\|_\Phi\\
& \ge \left\|\frac{x_k + y_k}{2}\right\|_\Phi \ge I_\Phi\left(\frac{x_k + y_k}{2}\right) \ge 1-c_k \to 1,
\end{align*}
as $k\to\infty$.  We also have by (\ref{eq:26}),
\[
\|f_k - g_k\|_{1,\Phi} =\|f_k - g_k\|_\Phi + \|\tilde{x}_k - \tilde{y}_k\|_\Phi \ge \|\tilde{x}_k - \tilde{y}_k\|_\Phi= 
\|x_k - y_k\|_\Phi \ge \gamma 
\]
for all $k\in\N$. It shows that $W^{1,\Phi}$ is not uniformly convex and the proof is finished.

\end{proof}
The next result follows from Theorems \ref{th:uniconmosp} and \ref{th:UC}.

\begin{corollary} \label{cor:ucmosmo}
Let $\Omega=(\alpha,\beta)$, $\infty<\alpha<\beta<\infty$. Let $\Phi$ satisfy ${\rm (V)}$. Then $L^\Phi$ is uniformly convex if and only if $W^{1,\Phi}$ is uniformly convex. This in turn is equivalent to $\Phi$ satisfying $\Delta_2$ and being uniformly convex.

\end{corollary}

\begin{corollary} \label{cor:ucSpx}
Let $\Omega=(\alpha,\beta)$, $\infty<\alpha<\beta<\infty$. The variable exponent Sobolev  space $W^{1,p(\cdot)}$ is uniformly convex if and only if $1<p^-\le p^+<\infty$.
\end{corollary}

\begin{proof}
It follows immediately from Corollaries \ref{cor:ucmosmo} and \ref{cor:uclpx}.

\end{proof}

\begin{corollary}  Let $\Omega=(\alpha,\beta)$, $\infty<\alpha<\beta<\infty$. If $\Phi$ is a double phase functional of the form {\rm(\ref{eq:dblfun})}, then the space $W^{1,\Phi}$ is uniformly convex
 if $1<p^-$ and $r^+<\infty$.
\end{corollary}
\begin{proof} It follows from Corollary \ref{cor:uclpr}.

\end{proof}

 In the same case of finite and non-atomic measure space,  $\Phi$ is uniformly convex   if and only if $\varphi$ satisfies the following condition \cite{Kam1982},
\begin{align}\label{eq:ucorlicz}
    \forall\varepsilon>0 \ \exists\delta>0 \ \exists u_0\geq 0\forall u, \ v\geq u_0  \ \ &|u-v|\geq\max\{u,v\}\\ \notag
    &\implies \varphi\left(\frac{u+v}{2}\right)\leq \frac{1-\delta}{2}(\varphi(u)+\varphi(v)).
\end{align}
In paper \cite{CHZ} the authors gave a characterization of uniform convexity  of $W^{1,\varphi}$ under additional assumption that $\varphi$ is a $N$-function. Moreover, their methods are very specific for Orlicz functions only,   and they not applicable in the case of $MO$ functions.  Notice also, that for an Orlicz function $\varphi$  the condition (V) is always satisfied on $\Omega = (\alpha, \beta)$, that is the Voltera operator is bounded on $L^\varphi$  (Theorem \ref{th:12}). By Theorem \ref{th:UC} and the above remarks we arrive at the following result.

\begin{corollary}
Let $\Omega=(\alpha,\beta) $ where $-\infty<\alpha<\beta<\infty$. For an Orlicz function $\varphi$, the Orlicz-Sobolev space  $W^{1,\varphi}$ is uniformly convex if and only if $\varphi$ satisfies $\Delta_2^\infty$ and is uniformly convex in the sense of {\rm (\ref{eq:ucorlicz})}.
\end{corollary}

\section{Superreflexivity and $B$-convexity of $W^{1,\Phi}$}\label{section 4.7}

  A Banach space $(X, \|\cdot\|)$ is said to be {\it $B$-convex} if there exist a $\delta>0$ and an integer $n\ge 2$ such that for any $x_1,\dots,x_n\in X$ we can choose $\epsilon = \{\epsilon_k\}_{k=1}^n$, $\epsilon_k = \pm 1$, in such a way that
\[
\left\|\frac1n \sum_{k=1}^n \epsilon_k x_k \right\| \le (1-\delta) \max_{1\le k\le n} \|x_k\|.
\]
A Banach space $X$ is said to be {\it superreflexive} if every Banach space $Y$ which is finitely representable in $X$ is reflexive \cite{AK}.  A uniformly convex Banach space is superreflexive. Any superreflexive Banach space has a uniformly convex equivalent norm \cite[Problem 11.6]{AK}. 

\begin{lemma} \label{lem:unif} \cite[Lemma 1.1.5]{HK}
Let $(\Omega,\Sigma,\mu)$ be a $\sigma$-finite non-atomic measure space. Given a $MO$ function $\Phi$ on $\Omega$ if $\Phi^*$ satisfies condition  $\Delta_2$, then there exists a uniformly convex $MO$ function $\Psi$ equivalent to $\Phi$.
\end{lemma}

\begin{theorem} \label{th:BconvMOSob}
Let $\Omega=(\alpha,\beta)$, $-\infty<\alpha<\beta<\infty$ and let $\Phi$ satisfy condition ${\rm (V)}$. Then the following conditions are equivalent.
\begin{itemize}
\item[{\rm(i)}]
$W^{1,\Phi}$ is reflexive.
\item[{\rm(ii)}]
$W^{1,\Phi}$ is superreflexive.
\item[{\rm(iii)}]
$W^{1,\Phi}$ is $B$-convex.
\item[{\rm(iv)}]
Both $\Phi$ and $\Phi^*$ satisfy condition $\Delta_2$.
\end{itemize}

\end{theorem}
\begin{proof}
 (ii) $\Rightarrow$ (i) is clear. 

  (i) $\Rightarrow$ (iv)    If (i) is satisfied, that is the space $W^{1,\Phi}$ is reflexive, then by Theorem \ref{th:ref} both $\Phi$ and $\Phi^*$ satisfy $\Delta_2$, so (iv) holds. 

  (iv) $\Rightarrow$ (ii)  By the assumption that $\Phi^*\in \Delta_2$, in view of Lemma \ref{lem:unif}, there exists a uniformly convex function $\Psi$ equivalent to $\Phi$. Since $\Phi\in \Delta_2$, the function $\Psi\in \Delta_2$.  Now by Theorem \ref{th:uniconmosp}, the space $W^{1,\Psi}$ is uniformly convex and thus superreflexive \cite[Problem 11.6]{AK}. Since $\Psi$ is equivalent to $\Phi$, the spaces $W^{1,\Phi}$ and $W^{1,\Psi}$ coincide as sets with equivalent norms (see Theorem \ref{th:equivMO}).  It follows that  $W^{1,\Phi}$  as isomorphic to $W^{1,\Psi}$ is  superreflexive. 

  (iv) $\Rightarrow$ (iii)  By \cite[Example 3 (ii), p. 118]{G}, any uniformly convex space is  $B$-convex.  By  $\Phi^*\in \Delta_2$, in view of Lemma \ref{lem:unif}, there exists a uniformly convex function $\Psi$ equivalent to $\Phi$. Since $\Delta_2$ is preserved by equivalence, $\Psi\in \Delta_2$.  In view of Theorem \ref{th:uniconmosp}, the space $W^{1,\Psi}$ is uniformly convex and thus is $B$-convex.
In the same paper \cite{G}, in Corollary 6 it was proved that if Banach spaces $X$ and $Y$ are isomorphic, then $X$ is $B$-convex if and only if $Y$ is $B$-convex. It follows that $W^{1,\Phi}$ is $B$-convex as $W^{1,\Psi}$ and  $W^{1,\Phi}$ coincide as sets with equivalent norms.

  (iii) $\Rightarrow$ (iv) If $\Phi\notin \Delta_2$, then $W^{1,\Phi}$ contains a subspace isomorphic to $\ell^\infty$ by Theorem \ref{th:ellinftycopy}, and since $\ell^\infty$ is not $B$-convex \cite[Example 3, (iv)]{G}, it contradicts the assumption of $B$-convexity of $W^{1,\Phi}$. 

  If $\Phi^*\notin\Delta_2$ then $W^{1,\Phi}$ contains a subspace isomorphic to $\ell^1$ by Theorem \ref{th:copyl1suf2}. However $\ell^1$ is not $B$-convex \cite[Example 3, (iv)]{G} or \cite{DJT},  and so $W^{1,\Phi}$ can not be $B$-convex.

\end{proof}

  In the next result let $\Phi(x,t) = \varphi(t)$ for all $x\in \Omega$, $t\ge 0$. Then $\varphi$ is an Orlicz function and $L^\varphi$ is an Orlicz space. For $\Omega = (\alpha, \beta)$, $-\infty < \alpha < \beta < \infty$,  and an Orlicz function $\varphi$, the condition (V) is always satisfied   by  Theorem \ref{th:12}, and  the condition $\Delta_2$ achieves a simpler form (\ref{ineq:deltatwo}). By these  remarks and Theorem \ref{th:BconvMOSob} we get the following corollary in Orlicz-Sobolev spaces.

\begin{corollary}\label{BconvOrSob}
Let $\Omega=(\alpha,\beta)$, $-\infty<\alpha<\beta<\infty$ and $\varphi$ be an Orlicz function.  Then the following conditions are equivalent.
\begin{itemize}
\item[{\rm(i)}]
$W^{1,\varphi}$ is reflexive.
\item[{\rm(ii)}]
$W^{1,\varphi}$ is superreflexive.
\item[{\rm(iii)}]
$W^{1,\varphi}$ is $B$-convex.
\item[{\rm(iv)}]
Both $\varphi$ and $\varphi^*$ satisfy condition $\Delta_2^\infty$.
\end{itemize}

\end{corollary}

  Since the Voltera operator is always bounded on $L^{p(\cdot)}$  by Corollary \ref{cor:voltera-lp}, the next corollary is an immediate result from Theorem \ref{th:BconvMOSob}.
 
\begin{corollary}\label{BconvSobpx}
Let $\Omega=(\alpha,\beta)$, $\infty<\alpha<\beta<\infty$, and $\Phi(x,t) = \frac{t^{p(x)}}{p(x)}$, $1\le p(x) < \infty$ a.e. in $\Omega$.  Then the following conditions are equivalent.
\begin{itemize}
\item[{\rm(i)}]
$W^{1,p(\cdot)}$ is reflexive.
\item[{\rm(ii)}]
$W^{1,p(\cdot)}$ is superreflexive.
\item[{\rm(iii)}]
$W^{1, p(\cdot)}$ is $B$-convex.
\item[{\rm(iv)}]
 $1<p^- \le p^+ < \infty$.
\end{itemize}

\end{corollary}

\end{document}